\documentclass[11pt]{article}
\usepackage{amsfonts}
\usepackage{amssymb}
\usepackage{amsthm}
\usepackage{amsmath}
\usepackage{graphicx}
\usepackage{empheq}
\usepackage{indentfirst}
\usepackage{cite}
\usepackage{mathrsfs}
\usepackage{graphics}
\usepackage{enumerate}
\usepackage{color}

 \newtheorem{theorem}{Theorem}[section]
 \newtheorem{corollary}{Corollary}[section]
 \newtheorem{lemma}{Lemma}[section]
 \newtheorem{proposition}{Proposition}[section]
 \newtheorem{definition}{Definition}[section]
 \newtheorem{remark}{Remark}[section]
 \numberwithin{equation}{section}

\def\u{\tilde{u}}
\def\v{\tilde{v}}
\def\M{\mathcal{M}}

\allowdisplaybreaks[4]

\def\u{\tilde{u}}

\def\e{\varepsilon}
\def\d{\delta}

\newcommand{\beq}{\begin{equation}}
\newcommand{\eeq}{\end{equation}}
 \def\non{\nonumber }
\def\bea{\begin{eqnarray}}
\def\eea{\end{eqnarray}}

\def\d{\delta}

\begin{document}
\title{On a Repulsion Keller--Segel System with a Logarithmic Sensitivity}
\author{Jie Jiang\thanks{Innovation Academy for Precision Measurement Science and Technology, CAS,
		Wuhan 430071, HuBei Province, P.R. China,
\textsl{jiang@wipm.ac.cn, jiang@apm.ac.cn}.}}

\date{\today}

\maketitle

\begin{abstract} In this paper, we study the initial-boundary value problem of a repulsion Keller--Segel system with a logarithmic sensitivity modeling the reinforced random walk. By establishing an energy-dissipation identity, we  prove the existence of classical solutions in two dimensions as well as existence of weak solutions in the three-dimensional setting. Moreover, it is shown that the weak solutions enjoys an eventual regularity property, i.e., it becomes regular after certain time $T>0$. An exponential convergence rate toward the spatially homogeneous steady states is obtained as well. We adopt a new approach developed recently by the author \cite{J19} to study the eventual regularity. The argument is based on observation of the exponential stability of constant solutions in scaling-invariant spaces together with certain dissipative property of the global  solutions in the same spaces.	
	
\end{abstract}
\section{Introduction}
Chemotaxis is the movement of cells in response to a chemical stimulus. If the movement is toward a higher concentration of the chemical, the motion is called chemo-attraction (or positive chemotaxis), while it is called chemo-repulsion (or negative chemotaxis) if  such a movement is in the opposite direction. PDE systems characterizing chemo-attraction such as the classical Keller--Segel model has been widely studied in recent years. A significant feature of chemo-attraction system is the aggregation of mass and therefore, blowup may take place in finite time or infinite time. A number of contributions have been devoted to the  blowup behavior of solutions, see e.g., \cite{BBTW15,HW01,HW05,MizoSoup,Win13,JWZ18,Win18}.

In contrast, there are only a few results in the existing literature for the chemo-repulsion system.
Consider the Neumann boundary value problem of the prototype chemo-repulsion Keller--Segel model:
\begin{equation}\label{chemo0}
\begin{cases}
\rho_t-\Delta \rho=\nabla \cdot(\rho\nabla c),\qquad &x\in\Omega,\;t>0\\
\gamma c_t-\Delta c+c=\rho\qquad &x\in\Omega,\;t>0
\end{cases}
\end{equation}where $\Omega\subset\mathbb{R}^d$ with $d\geq1$ is a bounded domain with smooth boundary. Here, $\rho$ and $c$ denote the density of cells and the concentration of chemical signal, respectively. If $\gamma=0$, it was proved in \cite{Mock74,Mock75} that solutions exist globally, which are uniformly bounded and converge with an exponential rate to the steady state, i.e., no finite-time blowup can take place.  However, when $\gamma>0$, similar result is not easy to prove due to the lack of estimates on $c_t$.  Based on a Lyapunov functional method, it was shown in \cite{CLM08} that with large initial data classical solution exists globally when $d=2$ and converges with an exponential rate to the steady state. When $d=3,4$, only weak solutions  were obtained and convergence toward steady state was also proved in certain weak sense. Recently, we proved that the three-dimensional weak solutions to \eqref{chemo0} enjoy an eventual smoothness property and converge to  the steady state exponentially \cite{J19} (in fact with slight modification of the proof, this is also true in four dimensions). There we proposed a unified new method to discuss the eventual smoothness  as well as exponential stabilization of weak solutions for two different kinds of chemotaxis models including the above chemo-repulsion one \eqref{chemo0} . The main idea was invoked by the recent study by the author \cite{J18a,J18b} on exponential stability of spatially homogeneous solutions in critical Lebesgue spaces  together with the dissipative property of the global weak solutions in the same spaces.  On the other hand,  if the chemo-sensitivity is nonlinear in $\rho$, Tao \cite{Tao13} considered the Neumann boundary problem of system \eqref{chemo0} with the first equation replaced by $\rho_t-\Delta \rho=\nabla \cdot(f(\rho)\nabla c)$ on a bounded convex domain with $d\geq3$, where $f(\rho)\leq K(1+\rho)^m$. Global existence of classical solution and convergence toward steady state were established under the assumption $0<m<\frac{4}{d+2}$.

If the cells are assumed to respond to the changes of the logarithm of the chemical concentration, i.e., following the Weber--Fechner law, the chemotactic movement is inhibited by the high chemical concentration. Recently, the following chemo-repulsion model with logarithmic sensitivity was proposed in \cite{LS97,OS97} to model the reinforced random walk:
\begin{equation}
\begin{cases}\label{chemo1g}
\rho_t-D\Delta \rho=\nabla \cdot(\chi\rho\nabla \log c)\\
c_t-\e\Delta c=g(\rho,c).
\end{cases}
\end{equation} When $g(\rho,c)=\rho c-\mu c$ with $\mu>0$, by the Hopf--Cole transformation and a Lyapunov functional method, existence and longtime behavior of global classical solution was studied in the one-dimensional setting in \cite{TWW13}.

In  the present paper, we consider the case $g(\rho,c)=\rho-c$, i.e., the production of chemical is proportional to the density. More precisely, we consider the following initial boundary value problem
\begin{equation}\label{chemo1}
\begin{cases}
\rho_t-\Delta \rho=\chi\nabla \cdot(\frac{\rho}{c} \nabla c),\qquad &x\in\Omega,\;t>0\\
\gamma c_t-\Delta c+c=\rho,\qquad &x\in\Omega,\;t>0\\
\partial_\nu\rho=\partial_\nu c=0,\qquad &x\in\partial\Omega,\;t>0\\
\rho(x,0)=\rho_I(x),\;\; c(x,0)=c_I(x),\qquad & x\in\Omega
\end{cases}
\end{equation}with some $\chi,\gamma>0.$

We mention that the corresponding chemo-attraction  Keller--Segel model with logarithmic sensitivity (i.e., $\chi<0$ in \eqref{chemo1}) has been studied in many recent works. However, theoretical results are far from satisfactory. Roughly speaking, existence of global solutions or blowups seems to be determined by the size of $\chi$. Blowup solution was constructed only  in the radial symmetric case when $\gamma=0$,  $n\geq3$ and $-\chi>\frac{2n}{n-2}$ \cite{Nagai98}. On the other hand, there are several attempts on enlarging the admissible range of $\chi$ ensuring global existence and however, the threshold number is still unclear. We refer the readers to \cite{Anh19,FS2016} for a complete review of related results.


To formulate our results,  we need to introduce some notion and notations. For any $a\geq0,$ denote by $L^p_a(\Omega)$ ($1\leq p<\infty$) the closed convex subset of $L^p(\Omega)$ satisfying $\overline{w}\triangleq\frac{1}{|\Omega|}\int_\Omega w dx=a$ with $w\in L^p(\Omega).$ Note that if $a=0$, $L^p_0(\Omega)$ is a Banach spaces and the following Poincar\'e's inequality holds:
\begin{equation}\non
\|w\|_{L^p(\Omega)}\leq c\|\nabla w\|_{L^p(\Omega)},\qquad\text{for all}\;\;w\in L^p_0(\Omega)
\end{equation}and we denote $\lambda_1$  the first positive eigenvalue of the Neumann Lapacian operator such that
\begin{equation}\label{poin}
\lambda_1\|w\|_{L^2(\Omega)}^2\leq \|\nabla w\|^2_{L^2(\Omega)} ,\qquad\text{for all}\;\;w\in L^2_0(\Omega).
\end{equation}
Then we introduce the same notion of weak solutions to \eqref{chemo1} as in \cite{CLM08}.
\begin{definition}\label{defwk}
	A global weak solution of $\eqref{chemo1}$ is a pair of  functions 
	\begin{equation}
	(\rho,c)\in C([0,\infty);\text{weak}-L^1(\Omega;\mathbb{R}^2),\non
	\end{equation}such that $\rho(x,t)\geq0$, $c(x,t)>0$ in $(0,T)\times\overline{\Omega}$ for any $T>0$,
	\begin{equation}
	\nabla \rho,\nabla c,\frac{\rho\nabla c}{c}\in L^1((0,T)\times\Omega)\non
	\end{equation}and 
	\begin{equation}
	\int_\Omega(\rho(t)-\rho_I)\varphi dx+\int_0^t\int_\Omega(\nabla \rho+\frac{\rho}{c} \nabla c)\cdot\nabla \varphi dxds=0\non,
	\end{equation}
	\begin{equation}\non
	\int_\Omega (c(t)-c_I)\varphi dx+\int_0^t\int_\Omega(\nabla c\cdot\nabla \varphi-(\rho-c)\varphi)dxds=0
	\end{equation}for each $t\geq0$ and $\varphi\in W^{1,\infty}(\Omega).$
\end{definition}
Throughout this paper, for the sake of simplicity, we assume that the initial data $\rho_I$ and $c_I$ satisfy
\begin{equation}\label{ini}
\rho_I\in C^0(\overline{\Omega}),\;c_I\in C^1(\overline{\Omega}),\qquad\partial_\nu c_I=0\;\;\text{on}\;\;\partial\Omega,\qquad\rho_I\geq0,\;\rho_I\not\equiv0,\;\;c_I>0 \;\;\text{in}\;\;\overline{\Omega}.
\end{equation}

Now we are in a position to state our main results. 
The first is concerned with existence of global solutions.
\begin{theorem}\label{TH1}For any given initial data $(\rho_I,c_I)$ satisfying \eqref{ini}, we have
	\begin{enumerate}[(i)]
		\item if $d=2$,  problem \eqref{chemo1} permits a unique classical solution;
		\item if $d=3$,  problem \eqref{chemo1} has a global weak solution $(\rho,c)$ in the sense of Definition \ref{defwk}. Moreover, for any $T>0$, there holds
		\begin{equation}
			\rho\in L^{5/4}(0,T; W^{1,5/4}(\Omega))
		\end{equation}and
		\begin{equation}
			 c\in C([0,T]; L^2(\Omega))\cap L^{4/3}(0,T; H^2(\Omega))\cap L^{2}(0,T;W^{2,3/2}(\Omega)).
		\end{equation}
	\end{enumerate}
\end{theorem}
We remark that existence of weak solutions in the four-dimensional case is yet unknown since the estimates for $c$ provided by our energy-dissipation relation \eqref{energy} are weaker than those derived for system \eqref{chemo0} in \cite{CLM08}.

Regarding the longtime behavior, if the domain is {\it convex}, we may employ a Lyapunov functional to obtain the folllowing  eventual regularity and exponential stabilization of the weak solutions.
\begin{theorem}\label{TH2}
	Let  $\Omega\subset\mathbb{R}^3$  be a bounded convex domain with smooth boundary. For any given initial data $(\rho_I,c_I)$ satisfying \eqref{ini}, problem \eqref{chemo1} possesses a global weak solution $(\rho,c)$  in the sense of Definition \ref{defwk} and moreover, for any $0<\mu'<\lambda_1$, there is $\tau_0>0$ such that this solution is bounded and belongs to $C^{2,1}(\overline{\Omega}\times[\tau_0,\infty))$ and
	\begin{equation}
	\|\rho(t)-\M\|_{L^{\infty}(\Omega)}+\|\nabla c(t)\|_{L^{\infty}(\Omega)}\leq C e^{-\mu'(t-\tau_0)}\qquad \text{for}\;\; t\geq \tau_0
	\end{equation}and
	\begin{equation}
	\|c(t)-\M\|_{L^{\infty}(\Omega)}\leq C e^{-\min\{\mu',1\}(t-\tau_0)}\qquad \text{for}\;\; t\geq \tau_0
	\end{equation}
	where $\M\triangleq\frac{1}{|\Omega|}\int_\Omega\rho_I dx$ and $C>0$ depends only on $\mu'$, $\Omega$ and the initial data. 
\end{theorem}
\begin{corollary}
	If $d=2$ and $\Omega$ is convex,  the classical solution will converge to $(\M,\M)$ exponentially with the convergence rates given above.
\end{corollary}
As a byproduct of our approach, we obtain the exponential stability of spatially homogeneous solutions in the scaling-invariant spaces. Note that the convexity assumption on domain is not needed here.
\begin{theorem}\label{TH3}
		Let $d\geq2.$ For any given constants $\M>0$ and $\mu'<\lambda_1$, there exists $\varepsilon_0>0$ depending on $\M$, $d$, $\mu'$ and $\Omega$ such that for any  initial data $(\rho_I,c_I)\in \left(C(\overline{\Omega})\cap L^1_\M(\Omega)\right)\times C^1(\overline{\Omega})$ satisfying \eqref{ini} and  $\|\rho_I-\M\|_{L^{d/2}(\Omega)}+\|\nabla c_I\|_{L^d(\Omega)}+\|c_I-\M\|_{L^\infty(\Omega)}\leq\varepsilon_0$, problem \eqref{chemo1}  has a unique global classical solution such that
	\begin{equation}
	\|\rho(\cdot,t)-\M\|_{L^\infty(\Omega)}+\|\nabla c(\cdot,t)\|_{L^\infty(\Omega)}\leq Ce^{-\mu' t}\qquad\text{for all}\;\;t\geq1
	\end{equation} 
	and
	\begin{equation}
	\|c(t)-\M\|_{L^{\infty}(\Omega)}\leq C e^{-\min\{\mu',1\}t}\qquad \text{for all}\;\; t\geq 1
	\end{equation}with some $C>0.$
\end{theorem}
We would like to mention that exponential stability of constant solutions for chemo-attraction Keller--Segel equations in scaling-invariant spaces has been recently studied by the author in \cite{J18a,J18b}. The proof was based on certain delicate $L^p-L^q$ decay estimates for the corresponding linearized semigroups together with a modification of a one-step contraction argument borrowed from \cite{Win10,Cao}.

The main strategy of our proof for Theorem \ref{TH2}--Theorem \ref{TH3} consists of two steps. Thanks to an energy-dissipation inequality, we first prove that for any small neighborhood of the spatially homogeneous steady state in the topology of scaling-invariant spaces, there always exists a waiting time at which the solution will fall into this neighborhood. Second, we establish the stability result of the constant steady states in the scaling-invariant spaces. Then once the solution falls into a sufficiently small neighborhood of  the constant steady states, the stability implies that it will stay  there forever.

For the reader's convenience, we briefly sketch the idea  here since our method is not conventional. Assuming we already have Theorem \ref{TH3}, then by the energy-dissipation inequality established in Lemma \ref{lya} below, a global solution $(\rho,c)$ on a convex domain satisfies
\begin{equation}\label{dissipation}
\int_0^\infty\int_\Omega\left(\frac{|\nabla\rho|^2}{\rho}+c|\nabla^2\log c|^2+\frac{|\nabla c|^2}{c}\right) dx dt<C^*
\end{equation} with $C^*$ depending only on  the initial data and $\Omega$. Then by some embeddings established in Lemma \ref{decay00}, there holds
\begin{equation}\label{idea4}
\int_0^\infty\left(\|\rho-\M\|^2_{L^{3/2}(\Omega)}+\|\nabla c\|_{L^3(\Omega)}^2+\|c-\M\|_{L^\infty(\Omega)}^2\right)dt\leq C^*.
\end{equation}
Apparently, this inequality  implies some dissipative property of the solution in the scaling-invariant spaces, e.g., there is a time sequence $t_k\rightarrow+\infty$ such that $\|\rho(t_k)-\M\|^2_{L^{3/2}(\Omega)}+\|\nabla c(t_k)\|_{L^3(\Omega)}^2+\|c(t_k)-\M\|_{L^\infty(\Omega)}^2\rightarrow0$. Then, we can regard the solution as a trajectory originating from $(\rho(t_k),c(t_k))$ for sufficiently large $k$ and in view of the stability in the scaling-invariant space $L^{3/2}\times {\dot{W}}^{1,3}\cap L^\infty$, we may anticipate the eventual smoothness as well as the exponential convergence.

However  the proof is not straightforward along the above idea when we deal with the {\it weak}  solutions due to the lack of regularity. To overcome this difficulty, one needs to introduce proper approximation problems and  establish the exponential stability result as in Theorem \ref{TH3} which holds uniformly with respect to a family of approximating systems. To this aim, we need  to make use of  an approximation scheme called volume-filling effect introduced  in \cite{CLM08} (see also \cite{HP09}) which preserves the scaling-invariant structure of the original system  with an extra perturbation term.

The rest of this paper is organized as follows. In Section 2, we introduce the local existence of classical solutions and recall some useful lemmas which are needed in the subsequent proof. In Section 3, based on an energy-dissipation relation, we prove the existence of global classical solutions in 2D  via standard energy method and the existence of weak solutions in 3D  by proper approximations. Then in Section 4, we prove the eventual regularity of the 3D weak solutions along the idea sketched above.

\section{Local Well-posedness and Some Preliminaries}
For each $\e\geq0$, we introduce the approximation problem of \eqref{chemo1} as follows.
\begin{equation}\label{chemo1a}
(AP)\quad\begin{cases}
\partial_t\rho_\e-\Delta \rho_\e=\nabla \cdot(\frac{\rho_\e(1-\e\rho_\e)}{c_\e} \nabla c_\e),\qquad &x\in\Omega,\;t>0\\
\partial_t c_\e-\Delta c_\e+c_\e=\rho_\e,\qquad &x\in\Omega,\;t>0\\
\partial_\nu\rho_\e=\partial_\nu c_\e=0,\qquad &x\in\partial\Omega,\;t>0\\
\rho_\e(x,0)=\rho_I(x),\;\; c_\e(x,0)=c_I(x),\qquad & x\in\Omega.
\end{cases}
\end{equation}
For the approximation problem, we have the following local existence result.	\begin{theorem}\label{locala}
		For any given $(\rho_I, c_I)$ satisfying \eqref{ini}, system \eqref{chemo1a} has a unique local classical solution
		\begin{equation}\non
		(\rho_\e,c_\e)\in C(\overline{\Omega}\times [0,t_\e^+);\mathbb{R}^2)\cap C^{2,1}(\overline{\Omega}\times (0,t_\e^+);\mathbb{R}^2)
		\end{equation}and $\rho_\e(x,t)\geq0, c_\e(x,t)>0$ for each $(x,t)\in \overline{\Omega}\times [0,t_\e^+)$, $t_\e^+$ denoting the maximal existence time. Moreover, $\|\rho_\e(t)\|_{L^1(\Omega)}=\|\rho_I\|_{L^1(\Omega)}$ and $\|c_\e(t)\|_{L^1(\Omega)}=e^{-t}\|c_I\|_{L^1(\Omega)}+(1-e^{-t})\|\rho_I\|_{L^1(\Omega)}.$
		
		If there is a function $\omega:(0,\infty)\rightarrow(0,\infty)$ such that  for each $T>0$,
		\begin{equation}\label{cri0}
		\|\rho_\e(t)\|_{L^\infty(\Omega)}\leq \omega(T), \qquad 0<t<\min\{T,t_\e^+\},
		\end{equation} then $t_\e^+=+\infty.$ In particular, if $\e\in(0,\zeta_0]$ with $1/{\zeta_0}=\|\rho_I\|_{L^\infty(\Omega)}$, then  $\rho_\e\leq1/{\e}$ and thus $t_\e^+=+\infty.$	
	\end{theorem}
\begin{proof}
	The local existence result can be obtained in the way as done in \cite[Lemma 3.1]{BBTW15} for a more general setting since the boundedness of the sensitivity term $\frac{\rho_\e(1-\e\rho_\e)}{c_\e}$ follows from the strictly positivity of $c_\e$ on $\overline{\Omega}\times(0,t_\e^+)$ given by Lemma \ref{unilow} below. The positivity of $\rho_\e,c_\e$ and the property $\rho_\e\leq1/\varepsilon$ when $\varepsilon\in(0,\zeta_0]$ are given in \cite{CLM08}. We omit the detail here.
	
	If $t_\e^+<\infty,$ due to criterion \cite[Eqn. (3.3)]{BBTW15}, for any $q>d$, there holds 
	\begin{equation}\label{cri1}
		\limsup\limits_{t\nearrow t_\e^+}\|\rho_\e(\cdot,t)\|_{L^\infty(\Omega)}+\|c_\e(\cdot,t)\|_{W^{1,q}(\Omega)}=\infty.
	\end{equation}However, if \eqref{cri0} holds, the standard theory for parabolic equations implies that  $\|c_\e\|_{L^\infty(\Omega)}<\infty$ and by Lemma \ref{lmpq},
	\begin{align}
		\|\nabla c_\e(t)\|_{L^q(\Omega)}\leq&\|\nabla e^{(\Delta-1)t} c_I\|_{L^q(\Omega)}+C\int_0^te^{-(\lambda_1+1)(t-s)}(1+(t-s)^{-\frac12})\|\rho(s)\|_{L^q(\Omega)}ds\non\\
		\leq&C\|c_I\|_{W^{1,\infty}(\Omega)}+C\int_0^T(1+(t-s)^{-\frac12})ds\non\\
		<&\infty
	\end{align}which by \eqref{cri1} implies that $t_\e^+=\infty.$ This completes the proof.
\end{proof}
Since $c_I$ is strictly positive, one has a uniform-in-time (and also uniform-in-$\e$) strictly positive lower bound for $c_\e$ due to the following lemma (see, e.g., \cite[Lemma 3.1]{Fujie}).
\begin{lemma}\label{unilow}
	There is $c_*>0$ depending only on $\inf\limits_{x\in\Omega}c_I$, $\Omega$ and $\|\rho_I\|_{L^1(\Omega)}$ such that
	\begin{equation}
		\inf\limits_{x\in\Omega} c_\e(x,t)\geq c_*>0.
	\end{equation}
\end{lemma}

Next, we recall some lemmas of analysis which will be used in the sequel. The first one is the following result of analysis given in  \cite[Lemma 3.3]{WinCPDE}.
\begin{lemma}
	\label{lm1}
	Let $h\in C^1(0,\infty)$ be positive and let $\Theta(s):=\int_1^s\frac{d\sigma}{h(\sigma)}$ for $s>0$. Assume that $\Omega$ is a smooth bounded domain in $\mathbb{R}^d$ with $d\geq1$. Then for any positive function $\varphi\in C^2(\overline{\Omega})$ fulfilling $\frac{\partial\varphi}{\partial\nu}=0$ on $\partial\Omega$, there holds
	\begin{equation}
	\int_\Omega \frac{h'(\varphi)}{h^3(\varphi)}|\nabla\varphi|^4 dx
	\leq (2+\sqrt{d})^2\int_\Omega \frac{h(\varphi)}{h'(\varphi)}|\nabla^2\Theta(\varphi)|^2 dx.\nonumber
	\end{equation}
\end{lemma}
Then we recall the following two results given in \cite[Lemma 2.3 \& Lemma 2.4]{J19}.
\begin{lemma}\label{lmanalysis2}
	For any positive function $\varphi\in C^2(\overline{\Omega})$, there holds
	\beq\label{n2}
	2\frac{\Delta \sqrt{\varphi}}{\sqrt{\varphi}}=\frac{\Delta \varphi}{\varphi}-\frac{|\nabla \varphi|^2}{2\varphi^2},\quad 2\varphi\nabla (\frac{\Delta\sqrt{\varphi}}{\sqrt{\varphi}})=\nabla \cdot (\varphi \nabla^2\log \varphi)
	\eeq 	
	and moreover, if $\frac{\partial\varphi}{\partial\nu}=0$ on $\partial\Omega$, we have
	\begin{equation}\label{n9}
	\int_\Omega \left[-\nabla \varphi\cdot\nabla \left(\frac{\Delta \varphi}{\varphi}\right)-\frac12\int_\Omega |\nabla \log\varphi|^2\Delta \varphi\right]dx=-\frac12\int_{\partial\Omega}\frac1\varphi \frac{\partial}{\partial\nu}|\nabla\varphi|^2ds+\int_\Omega\varphi|\nabla^2\log\varphi|^2dx.
	\end{equation}
\end{lemma}

\begin{lemma}\label{lm6} For any positive function $\varphi\in C^2(\overline{\Omega})$ with $\Omega\subset\mathbb{R}^d$ fulfilling $\frac{\partial\varphi}{\partial\nu}=0$ on $\partial \Omega$, there is a positive constant $C$ depending only on $d$ such that
	\begin{equation}\nonumber
	\int_\Omega \left(\frac{|\Delta \varphi|^2}{\varphi}+|\Delta\sqrt{\varphi}|^2+\frac{|\nabla \varphi|^4}{\varphi^3}\right)dx\leq C\int_\Omega\varphi|\nabla^2\log\varphi|^2dx.
	\end{equation}
\end{lemma}

The following result indicates that a lower-order perturbation to a sectorial operator is still a sectorial operator \cite{Zheng04}.
\begin{lemma}\label{pertur1}
	Suppose that $A$ is a sectorial operator and $B$ is a linear operator with $D(A)\subset D(B)$ such that for any $x\in D(A)$, there holds
	\begin{equation}\nonumber
	\|Bx\|\leq \varepsilon \|Ax\|+K_\varepsilon \|x\|
	\end{equation}where $\varepsilon>0$ is an arbitrary small constant and $K_\varepsilon$ is a positive constant depending on $\varepsilon.$ Then $A+B$ is sectorial.
\end{lemma}

The next lemma presents an estimate for frequently used integrals throughout this paper, the proof of which can be found in \cite{Win10,J18b}.
\begin{lemma}\label{lmint}
	Suppose $0<\alpha<1$, $0<\beta<1$, $\gamma>0$, $\delta>0$ and $\gamma\neq\delta$. Then there holds
	\begin{equation}\nonumber
	\int_0^t(1+(t-s)^{-\alpha})e^{-\gamma(t-s)}(1+s^{-\beta})e^{-\delta s}ds\leq C(\alpha,\beta,\delta,\gamma)(1+t^{\min\{0,1-\alpha-\beta\}})e^{-\min\{\gamma,\delta\}t}
	\end{equation}for all $t>0$,  where $C(\alpha,\beta,\delta,\gamma)=C\cdot(\frac{1}{|\delta-\gamma|}+\frac{1}{1-\alpha}+\frac{1}{1-\beta})$ with  $C>0$ being a generic constant when $0<t\leq 1$ or when $t>1$ and $\alpha+\beta\geq1$, while when $t>1$ and $\alpha+\beta<1$, the constant $C$ may also depend on $\left(\frac{2(1-\alpha-\beta)}{\alpha(\delta-\gamma)}\right)^{\frac{1-\alpha-\beta}{\alpha}}$.
\end{lemma}
Next, we recall the important $L^p-L^q$ estimates for the Neumann heat semigroup on bounded domains (see e.g., \cite{Cao,Win10}).
\begin{lemma}\label{lmpq}
	Suppose $\{e^{t\Delta}\}_{t\geq0}$ is the Neumann heat semigroup in $\Omega$, and $\lambda_1>0$ denote the first nonzero eigenvalue of $-\Delta$ in $\Omega$ under Neumann boundary conditions. Then there exist $k_1,..., k_4>0$ which only depend on $\Omega$ such that the following properties hold:
	\begin{enumerate}[(i)]
		\item If $1\leq q\leq p\leq \infty,$ then
		\begin{equation}
		\|e^{t\Delta}w\|_{L^p(\Omega)}\leq k_1(1+t^{-\frac{d}{2}(\frac1q-\frac1p)})e^{-\lambda_1t}\|w\|_{L^q(\Omega)}\qquad\text{for all}\;\;t>0
		\end{equation}for all $w\in L^q_0(\Omega)$;
		\item  If $1\leq q\leq p\leq \infty,$ then
		\begin{equation}
		\|\nabla e^{t\Delta}w\|_{L^p(\Omega)}\leq k_2(1+t^{-\frac12-\frac{d}{2}(\frac1q-\frac1p)})e^{-\lambda_1t}\|w\|_{L^q(\Omega)}\qquad\text{for all}\;\;t>0
		\end{equation}for each $w\in L^q(\Omega)$;
		\item If $2\leq q\leq p<\infty,$ then
		\begin{equation}
		\|\nabla e^{t\Delta}w\|_{L^p(\Omega)}\leq k_3e^{-\lambda_1 t}(1+t^{-\frac{d}{2}(\frac1q-\frac1p)})\|\nabla w\|_{L^q(\Omega)}\qquad\text{for all}\;\;t>0
		\end{equation}for all $w\in W^{1,p}(\Omega)$;
		\item  If $1<q\leq p\leq \infty,$ then
		\begin{equation}
		\|e^{t\Delta}\nabla \cdot w\|_{L^p(\Omega)}\leq k_4(1+t^{-\frac12-\frac{d}{2}(\frac1q-\frac1p)})e^{-\lambda_1 t}\|w\|_{L^q(\Omega)}\quad\text{for all}\;\;t>0
		\end{equation}for any $w\in (W^{1,p}(\Omega))^d.$
	\end{enumerate}
\end{lemma}
Last, we recall the following auxiliary result given in \cite[Lemma 2.8]{J19} which unveils the dissipative property of the solutions.
\begin{lemma}\label{lm0}
	Suppose that $0\leq f_n(t)\in L^1(0,\infty)$ with $n\in\mathbb{N}$ is a sequence of functions such that for any $n$, there holds
	\begin{equation}\label{lm0a}
	\int_0^\infty f_n(s)ds\leq C
	\end{equation}with $C>0$ independent of $n$. Then, 
for arbitrary $\sigma>0$, there is $k_\sigma>0$ such that for any $n$, there is a time $t^n_{\sigma}\in(0, k_\sigma)$ such that $f_{n}(t^n_{\sigma})\leq \sigma$.
\end{lemma}

\section{Existence of Global Solutions}
\subsection{The Energy-Dissipation Relation}
Now we introduce the following energy-dissipation relation for the approximation system \eqref{chemo1a} which plays a key role in deriving the global existence for our problem.
\begin{lemma}\label{en} For any $\e\geq0$, the solution $(\rho_\e,c_\e)$ satisfies the following relation
	\begin{align}\label{energy}
		\frac{d}{dt} E_\e(\rho_\e(t),c_\e(t))+&\int_\Omega\frac{|\nabla \rho_\e|^2}{\rho_\e(1-\e\rho_\e)}dx
	+\int_\Omega c_\e|\nabla^2\log c_\e|^2dx\non\\
	&\qquad+\int_\Omega \frac{\rho_\e|\nabla c_\e|^2}{2c_\e^2}dx+\int_\Omega \frac{|\nabla c_\e|^2}{2c_\e} dx=\frac12\int_{\partial\Omega}\frac1{c_\e} \frac{\partial}{\partial\nu}|\nabla c_\e|^2ds
	\end{align}
where $E_\e$ is given by
\begin{equation}\nonumber
	E_\e(\rho,c)=\int_\Omega \left(\rho\log\rho+\frac{1}{\e}(1-\e\rho)\log(1-\e\rho)+2 |\nabla\sqrt{c}|^2\right)dx
\end{equation}	if $\e>0$ and
\begin{equation}\nonumber
	E_0(\rho,c)=\int_\Omega \left(\rho\log\rho+2 |\nabla\sqrt{c}|^2\right)dx.
\end{equation}
\end{lemma}
\begin{proof}
	First, a multiplication of the first equation  with $\log \rho_\e-\log(1-\e\rho_\e)$ together with an integration over $\Omega$ yields that
	\begin{align}\label{lya1}
	\frac{d}{dt}\int_\Omega \left(\rho_\e\log\rho_\e+\frac{1}{\e}(1-\e\rho_\e)\log(1-\e\rho_\e)\right) dx+\int_\Omega\frac{|\nabla \rho_\e|^2}{\rho_\e(1-\e\rho_\e)}dx=-\int_\Omega\frac{\nabla c_\e\cdot\nabla \rho_\e}{c_\e} dx.
	\end{align}
	On the other hand, multiplying the second equation by $-\frac{\Delta c_\e}{c_\e}$ and integrating by parts, we obtain that
	\begin{align}
	-\int_\Omega \Delta c_\e\partial_t\log c_\e dx-\int_\Omega\nabla c_\e\cdot\nabla \left(\frac{\Delta c_\e}{c_\e}\right)dx=\int_\Omega \frac{\nabla \rho_\e\cdot\nabla c_\e}{c_\e} dx-\int_\Omega\rho_\e\frac{|\nabla c_\e|^2}{c_\e^2}dx\nonumber
	\end{align} 
	where by integration by parts, we infer that
	\begin{align}
	-\int_\Omega \Delta c_\e\partial_t\log c_\e dx=&\int_\Omega c_\e\nabla \log c_\e\cdot\partial_t\nabla \log c_\e dx\non\\
	=&\int_\Omega\frac{c_\e}2\frac{\partial}{\partial t}|\nabla\log c_\e|^2dx\non\\
	=&\frac12\frac{d}{dt}\int_\Omega c_\e |\nabla\log c_\e|^2dx-\frac12\int_\Omega |\nabla\log c_\e|^2(\Delta c_\e+\rho_\e-c_\e)dx\non\\
	=&\frac12\frac{d}{dt}\int_\Omega c_\e |\nabla\log c_\e|^2dx-\frac12\int_\Omega |\nabla\log c_\e|^2\Delta c_\e dx\non\\
	&\qquad-\int_\Omega\left(\frac{\rho_\e|\nabla c_\e|^2}{2c_\e^2}- \frac{|\nabla c_\e|^2}{2c_\e}\right) dx.\non
	\end{align}
Then relation \eqref{energy} follows from Lemma \ref{lmanalysis2} and the above identities. This completes the proof. 
\end{proof}	
For convex domains, the boundary integration term on the right hand-side of relation \eqref{energy} is non-positive and hence can be neglected. In contrast, when the domain is non-convex, we need a lemma of analysis \cite[Lemma 4.2]{MizoSoup} to deal with this boundary integration and thus an application of \cite[Lemma 2.4]{JWZ15} indicates that for any $\delta>0$, there holds
	\beq
	\begin{split}
		&\frac{1}{2}\left|\int_{\partial\Omega}\frac{1}{c_\e}\frac{\partial}{\partial\nu}|\nabla c_\e|^2ds\right| \\
		&\ \ \leq \d\int_{\Omega}c_\e|\Delta\log c_\e|^2dx+\d\int_{\Omega}\frac{|\nabla c_\e|^4}{c_\e^3}dx+C_{\d}\|c_\e\|_{L^1(\Omega)}
	\end{split}\label{bde1}
	\eeq where $C_\d$ is a constant that may depend on $\Omega$ and $\d$, but not on $c_\e$. Thus an application of  Lemma \ref{lm6} gives
	\begin{lemma}\label{lya2}There holds
\begin{align}\label{lya3}
&\frac{d}{dt}E_\e(\rho_\e(t),c_\e(t))+\int_\Omega\frac{|\nabla \rho_\e|^2}{\rho_\e(1-\e\rho_\e)}dx\non\\
&\qquad\qquad+\frac12\int_\Omega c_\e|\nabla^2\log c_\e|^2dx+\int_\Omega \frac{\rho_\e|\nabla c_\e|^2}{2c_\e^2}dx+\int_\Omega \frac{|\nabla c_\e|^2}{2c_\e} dx\leq C\| c_\e\|_{L^1(\Omega)}
\end{align}with $C>0$ depending on $\Omega$ and $d$ only.
	\end{lemma}

As a consequence of Lemma \ref{lya2}, one can easily derive the uniform-in-time estimates as follows.
\begin{lemma}\label{boun}
	For any $\e\in[0,\zeta_0]$ and $t\in[0,t_\e^+)$, there is $\kappa_0>0$ depending only on $\Omega$ and $\|\rho_I\|_{L^1(\Omega)}$ such that the solution $(\rho_\e,c_\e)$ satisfies
	\begin{align}\label{bound00}
	\int_\Omega\left(\rho_\e(t)|\log\rho_\e(t)|+2|\nabla \sqrt{c_\e(t)}|^2\right)dx\leq C
	\end{align}	and
	\begin{align}\label{bound0}
		\int_0^te^{-\kappa_0(t-s)}\int_\Omega\left(\frac{|\nabla \rho_\e|^2}{\rho_\e}+ c_\e|\nabla^2\log c_\e|^2+\frac{\rho_\e|\nabla c_\e|^2}{c_\e^2}+ \frac{|\nabla c_\e|^2}{c_\e} \right)dxds\leq C
	\end{align}where $C$ depends on the initial data and $\Omega$.
\end{lemma}
\begin{proof}
	First, in the same manner as in \cite[Lemma 2.4]{WinARMA}, by an elementary inequality
	\begin{equation}\nonumber
	\xi\log\xi\leq \frac{1}{p(p-1)}\xi^p,\qquad\forall\xi>0
	\end{equation}with any fixed $p\in(1,2)$, we deduce by the Gagliardo-Nirenberg inequality that for all $t>0$
	\begin{align}\nonumber
	\int_\Omega\rho_\e\log\rho_\e\leq&  \frac{1}{p(p-1)}\int_\Omega \rho_\e^p=\frac{1}{p(p-1)}\|\sqrt{\rho_\e}\|^{2p}_{L^{2p}(\Omega)}\non\\
	\leq& C\|\nabla \sqrt{\rho_\e}\|^{d(p-1)}\|\sqrt{\rho_\e}\|^{d-dp+2p}+C\|\sqrt{\rho_\e}\|^{2p}\nonumber
	\end{align}for $1<p<\min\{2,\frac{d}{d-2}\}$. Hence we obtain by the conservation of mass that
	\begin{equation}\nonumber
	\int_\Omega\rho_\e\log \rho_\e\leq C\|\nabla \sqrt{\rho_\e}\|^{d(p-1)}+C\qquad \forall \;t>0
	\end{equation}with $C>0$ depends on $\Omega$, $p$ and $\|\rho_I\|_{L^1(\Omega)}$ only. Now, we may pick $p\in(1,\frac{d+2}{d})$ such that
	\begin{equation}\nonumber
	\int_\Omega\rho_\e\log \rho_\e\leq C\|\nabla \sqrt{\rho_\e}\|^2+C\qquad \forall \;t>0.
	\end{equation}
	Next in view of following the elementary inequalities
	\begin{equation}\label{help1}
	\xi+\frac1{\e}(1-\e\xi)\log(1-\e\xi)\geq0\qquad \text{for}\;\;\xi\in[0,1/\e],
	\end{equation}
	\begin{equation}\label{help2}
	\xi\log\xi\geq-\frac{1}{e}\qquad \text{for}\;\;\xi\in[0,1]
	\end{equation}and
	\begin{equation}\label{help3}
	\frac{|\nabla \rho_\e|^2}{\rho_\e(1-\e\rho_\e)}\geq\frac{|\nabla \rho_\e|^2}{\rho_\e},
	\end{equation}
	 we infer that
	 \begin{align}
	 E_\e(\rho_\e,c_\e)\geq&\int_\Omega\rho_\e\log\rho_\e-\rho_\e+2|\nabla \sqrt{c_\e(t)}|^2\non\\
	 \geq&\int_\Omega\left(\rho_\e|\log\rho_\e|+2|\nabla \sqrt{c_\e(t)}|^2\right)dx-(\|\rho_I\|_{L^1(\Omega)}+\frac{2|\Omega|}{e})\nonumber
	 \end{align}and in addition,	 
	\begin{align}\label{unies9}
	E_\e(\rho_\e,c_\e)&\leq C\|\nabla \sqrt{\rho_\e}\|^2+\int_\Omega \frac{|\nabla c_\e|^2}{2c_\e}dx+C\non\\
	&\leq C\left(\int_\Omega \left(\frac{|\nabla \rho_\e|^2}{\rho_\e(1-\e\rho_\e)}+  \frac{|\nabla c_\e|^2}{2c_\e}\right)dx+1\right)
	\end{align}	where $C>0$ depends only on $\Omega$ and $\|\rho_I\|_{L^1(\Omega)}$. As a result,  there is $\kappa_0>0$ such that
	\begin{align}\label{diffineq}
	\frac{d}{dt}E_\e(\rho_\e(t),c_\e(t))+&\kappa_0E_\e(\rho_\e(t),c_\e(t))+C_{\kappa_0}\int_\Omega \left(\frac{|\nabla \rho_\e|^2}{\rho_\e(1-\e\rho_\e)}+ \frac{|\nabla c_\e|^2}{2c_\e}\right)dx\non\\
	&+\frac12\int_\Omega c_\e|\nabla^2\log c_\e|^2dx+\int_\Omega \frac{\rho_\e|\nabla c_\e|^2}{2c_\e}dx\leq C+C\|c_\e\|_{L^1(\Omega)}
	\end{align}with $C>0$ depending only on $\Omega$ and $\|\rho_I\|_{L^1(\Omega)}$.
	Then the proof is completed by solving the above differential inequality.
\end{proof}
As  mentioned before, if the domain is convex, the boundary integration term in \eqref{energy} is non-positive. Thus, we have
\begin{lemma}\label{lya}
	If $\Omega$ is convex, then $E_\e$ serves as a Lyapunov functional and there holds:
	\begin{align}\label{lya0}
	\frac{d}{dt} E_\e(\rho_\e(t),c_\e(t))+&\int_\Omega\frac{|\nabla \rho_\e|^2}{\rho_\e(1-\e\rho_\e)}dx
	+\int_\Omega c_\e|\nabla^2\log c_\e|^2dx\non\\
	&\qquad+\int_\Omega \frac{\rho_\e|\nabla c_\e|^2}{2c_\e^2}dx+\int_\Omega \frac{|\nabla c_\e|^2}{2c_\e} dx\leq 0
	\end{align} In consequence, for $\e\in(0,\zeta_0]$ there holds
	\begin{equation}\label{timeint}
	\int_0^\infty\int_\Omega\left(\frac{|\nabla \rho_\e|^2}{\rho_\e}+c_\e|\nabla^2\log c_\e|^2+ \frac{\rho_\e|\nabla c_\e|^2}{c_\e^2}+ \frac{|\nabla c_\e|^2}{c_\e}\right)dxdt\leq C
	\end{equation}where $C$ depends on $\Omega$ and the initial data.
\end{lemma}
\begin{remark}
	The convexity assumption seems to be essential to derive the convergence  toward equilibrium since the above Lyapunov functional is needed. In \cite{J19}, this assumption was successfully removed for a  chemo-attraction Keller--Segel model with consumption of chemoattractants since there $\|c_\e\|_{L^1(\Omega)}$ vanishes as time goes to infinity. However in our case, the trick used in \cite{J19} fails due to the lack of certain decay property of $c_\e.$
\end{remark}

Next we show the uniqueness of solutions to the corresponding stationary problem which reads
\begin{equation}\label{chemo1s}
\begin{cases}
-\Delta \rho_s=\nabla \cdot(\frac{\rho_s}{c_s} \nabla c_s),\qquad &x\in\Omega\\
-\Delta c_s+c_s=\rho_s,\qquad &x\in\Omega\\
\partial_\nu\rho_s=\partial_\nu c_s=0,\qquad &x\in\partial\Omega\\
\int_\Omega \rho_s dx=\int_\Omega \rho_I dx\triangleq m>0.
\end{cases}
\end{equation}
\begin{lemma}
	The stationary problem \eqref{chemo1s} has a unique non-negative solution $(\rho_s,c_s)=(\M,\M)$ where $\M=m/|\Omega|$.
\end{lemma}
\begin{proof}
	On the one hand, in view of  the boundary conditions, we deduce from the first equation of \eqref{chemo1s} that
	\begin{equation}
	\nabla \rho_s+\rho_s\nabla \log c_s=0\qquad\text{in}\;\;\Omega.	\nonumber
	\end{equation}
	On the other hand, multiplying the second equation of \eqref{chemo1s} by $-c_s\Delta c_s$, integrating by parts and substituting the above equation into the resultant, we obtain that
	\begin{equation}\nonumber
	\int_\Omega c_s|\Delta c_s|^2dx+2\int_\Omega c_s|\nabla c_s|^2dx=-\int_\Omega \rho_s c_s\Delta c_sdx=\int_\Omega \left(\rho_s|\nabla  c_s|^2+c_s\nabla \rho_s \cdot\nabla c_s \right) dx=0
	\end{equation}which implies that $c_s\equiv \mathrm{const.}$ and hence the second equation of \eqref{chemo1s} together with the mass conservation indicate that $\rho_s=c_s=\M.$	
\end{proof}

\subsection{Proof of Global Existence}
Since a positive lower bound  is available for $c_\e$, with Lemma \ref{lm6} and Lemma \ref{boun} at hand, the argument resembles that in \cite{CLM08}. However, our estimates on $c_\e$ obtained in Lemma \ref{boun} are weaker than the case in \cite{CLM08} and hence we need more care in the proof of weak solutions. 
First, in order to show  existence of classical solutions when $d=2$, we need derive some further estimates.
\begin{lemma}Assume $d=2$ and $\e=0$. Let $p\geq2$ and $T>0$. Then there is a positive constant $C(p,T)$ depending only on $\Omega,T$ and the initial data such that
\begin{equation}\nonumber
	\|\rho(t)\|_{L^p(\Omega)}\leq C(p,T)\qquad\text{for}\;\;t\in[0,t_0^+)\cap[0,T].
\end{equation}
\end{lemma}
\begin{proof}Due to the Sobolev embedding inequality and the Poincar\'e inequality, there holds
\begin{equation}\nonumber
	\|\rho-\overline{\rho}\|^2\leq C\|\nabla \rho\|_{L^1(\Omega)}^2.
\end{equation}Therefore, by H\"older's inequality
\begin{align}
	\int_\Omega |\rho-\overline{\rho}|^2dx\leq C\left(\int_\Omega|\nabla \rho|dx\right)^2\leq C\left(\int_\Omega\frac{|\nabla \rho|^2}{\rho}dx\right)\left(\int_\Omega\rho dx\right).\nonumber
\end{align}
Thus, it follows from Lemma \ref{boun} that
\begin{equation}\nonumber
	\int_0^{t}e^{-\kappa_0(t-s)}	\|\rho-\overline{\rho}\|^2ds\leq C.
\end{equation}
On the other hand, an integration the second equation of \eqref{chemo1} over $\Omega$ yields that
\begin{equation}\label{mean}
	\overline{c}_t+\overline{c}=\overline{\rho}
\end{equation}
and a reduction of the above equation from the second equation in \eqref{chemo1} gives
\begin{equation}\label{reduction}
\partial_t(c-\overline{c})-\Delta c+c-\overline{c}=\rho-\overline{\rho}.
\end{equation}
Multiplying \eqref{reduction} by $-\Delta c$ and integrating by parts, we obtain that
\begin{equation}\nonumber
	\frac{1}{2}\frac{d}{dt}\|\nabla c\|^2+\|\Delta c\|^2+\|\nabla c\|^2=-\int_\Omega(\rho-\overline{\rho})\Delta cdx\leq \frac12\|\Delta c\|^2+\frac12\|\rho-\overline{\rho}\|^2
\end{equation}which yields to 
\begin{equation}\nonumber
	\frac{d}{dt}\|\nabla c\|^2+\kappa_0\|\nabla c\|^2+\|\Delta c\|^2\leq \|\rho-\overline{\rho}\|^2.
\end{equation}
Hence, by solving the above differential inequality, we obtain that
\begin{equation}\nonumber
	\|\nabla c(t)\|^2+\int_0^te^{-\kappa_0(t-s)}(\|\nabla c\|^2+\|\Delta c\|^2)ds\leq C
\end{equation}where $C$ depends on the initial data and $\Omega$ only.

Observing that $|\Delta\log c|^2\leq d|\nabla^2\log c|^2$ pointwisely and due to Lemma \ref{unilow}, we infer that
\begin{equation}\nonumber
	\|\Delta \log c\|^2\leq 2\int_\Omega |\nabla^2\log c|^2 dx\leq 2\|\frac{1}{c}\|_{L^\infty(\Omega)}\int_\Omega c|\nabla^2\log c|^2 dx\leq \frac{2}{c_*}\int_\Omega c|\nabla^2\log c|^2 dx\non
\end{equation}and hence for any $T>0$, we have
\begin{equation}\label{int0}
	\int_0^T\|\Delta \log c\|^2 ds\leq\frac{2}{c_*}\int_0^T\int_\Omega c|\nabla^2\log c|^2 ds\leq  C(T).
\end{equation}
Next, we multiply the first equation of \eqref{chemo1} by $(p+1)\rho^{p}$, integrate with respect to $x$ to obtain that
\begin{align}
	&\frac{d}{dt}\int_\Omega \rho^{p+1} dx+\frac{4p}{p+1}\int_\Omega|\nabla \rho^{(p+1)/2}|^2dx\non\\
	&=p\int_\Omega\rho^{p+1}\Delta\log c dx\non\\
	&\leq p\|\rho^{(p+1)/2}\|^2_{L^4(\Omega)}\|\Delta \log c\|\non\\
	&\leq Cp\| \rho^{(p+1)/2}\|_{H^1(\Omega)}\|\rho^{(p+1)/2}\|\|\Delta \log c\|\non\\
	&\leq \|\nabla \rho^{(p+1)/2}\|^2+\|\rho^{(p+1)/2}\|^2+Cp^2\|\rho^{(p+1)/2}\|^2\|\Delta \log c\|^2\nonumber
\end{align}
Then, applying the Gronwall inequality and thanks to \eqref{int0}, we obtain that
\begin{equation}\nonumber
	\int_\Omega \rho^{p+1}dx+\int_0^T\int_\Omega|\nabla \rho^{(p+1)/2}|^2dxds\leq C(T).
\end{equation}This completes the proof.
\end{proof}
\noindent\textbf{Proof of Theorem \ref{TH1}.}\\
Once we obtain the above lemma in the two-dimensional setting, we may use Moser's iteration technique \cite{Ali} to show that for every $T>0$,
\begin{equation}\nonumber
	\|\rho(t)\|_{L^\infty(\Omega)}+\|c(t)\|_{W^{1,\infty}(\Omega)}\leq C(T).
\end{equation} Thus, we deduce that $t_0^+=\infty$ and $(\rho,c)$ is a global classical solution (cf. \cite{CLM08}).

For the three-dimensional case, the global existence of weak solutions follows from a compactness argument. First, as shown in \cite[Lemma 4.1, Eqn. (18)]{CLM08}, 
\begin{equation}
	\int_0^T\|\rho_\e\|_{L^{3p/(3p-2)}(\Omega)}^pdt\leq C(T),\qquad\text{for}\;p\in[1,\infty]
\end{equation}and in particular,
\begin{equation}\label{r53}
	\int_0^T\int_\Omega \rho_\e^{\frac{5}{3}}dxdt\leq C(T).
\end{equation}
Therefore, there holds
\begin{align*}
	\int_0^T\int_\Omega |\nabla \rho_\e|^{5/4}dxdt\leq\left(\int_0^T\int_\Omega \frac{|\nabla \rho_\e|^2}{\rho_\e}dxdt \right)^{5/8}\left(\int_0^T\int_\Omega \rho_\e^{5/3}dxdt \right)^{3/8}\leq C(T).
\end{align*}

Second, an application of the three-dimensional Agmon inequality and the Poincar\'e inequality yields that
\begin{align}
\|\sqrt{c_\e}\|_{L^\infty}^2\leq&\|\sqrt{c_\e}\|_{H^2}\|\sqrt{c_\e}\|_{H^1}\non\\
\leq&C(\|\Delta\sqrt{c_\e}\|+\|\sqrt{c_\e}\|)(\|\nabla \sqrt{c_\e}\|+\|\sqrt{c_\e}\|)\non\\
\leq&C\|\Delta\sqrt{c_\e}\|+C\nonumber.
\end{align}
As a consequence, by Lemma \ref{lm6}, there holds
\begin{align*}
\|\Delta c_\e\|^2\leq \|c_\e\|_{L^\infty(\Omega)}\int_\Omega\frac{|\Delta c_\e|^2}{c_\e}dx\leq C\left(\int_\Omega c_\e|\nabla^2\log c_\e|^2dx\right)^{3/2}+C.
\end{align*}
Besides, we note that
\begin{align*}
\|\Delta c_\e\|_{L^{3/2}(\Omega)}\leq &\left(\int_\Omega 
\frac{|\Delta c_\e|^2}{c_\e}dx\right)^{1/2}\left(\int_\Omega c_\e^{3}dx\right)^{1/6}\non\\
\leq& C\left(\int_\Omega c_\e|\nabla^2\log c_\e|^2dx\right)^{1/2}\left(\int_\Omega|\nabla \sqrt{c_\e}|^2dx+\|\sqrt{c_\e}\|^2\right)^{1/2}\non\\
\leq& C\left(\int_\Omega c_\e|\nabla^2\log c_\e|^2dx\right)^{1/2}.
\end{align*}
On the other hand, in view of the lower bound given in Lemma \ref{unilow} and with the application of Lemma \ref{lm6}, one easily deduce that
\begin{equation}\nonumber
\|\log c_\e\|_{H^2(\Omega)}^2\leq \frac{C}{c_*}\int_\Omega c_\e|\nabla^2\log c_\e|^2 dx+C(\|c_\e\|_{L^1(\Omega)}^2+\|\frac{1}{c_\e}\|_{L^\infty(\Omega)}^2).
\end{equation}
Similarly, we infer that
\begin{equation}\label{nalog4}
\int_0^T\|\nabla \log c_\e\|_{L^4(\Omega)}^4\leq\frac{1}{c_*}\int_0^T\int_\Omega \frac{|\nabla c_\e|^4}{c_\e^3}dxdt\leq C\int_0^T\int_\Omega c_\e|\nabla^2\log c_\e|^2dxdt\leq C(T).
\end{equation} 
In summary, we obtain that
\begin{equation}\nonumber
\rho_\e  \;\text{are uniformly bounded in}\;\; L^{5/4}(0,T; W^{1,5/4}(\Omega))\cap L^{5/3}(\Omega\times(0,T)).
\end{equation}
\begin{equation}\nonumber
c_\e  \;\text{are uniformly bounded in}\;\; L^{4/3}(0,T; H^2(\Omega))\cap L^{2}(0,T;W^{2,3/2}(\Omega))
\end{equation}
\begin{equation}\nonumber
	\sqrt{c_\e} \;\text{are uniformly bounded in}\;\; L^{\infty}(0,T; H^1(\Omega))\cap L^{2}(0,T;H^2(\Omega))
\end{equation}
and 
\begin{equation}\nonumber
\log c_\e  \;\text{are uniformly bounded in}\;\; L^2(0,T; H^2(\Omega))\cap L^4(0,T; W^{1,4}(\Omega))
\end{equation}
which together with the second equation and Lemma \ref{lm6} indicate that
\begin{equation}\nonumber
\partial_t c_\e  \;\text{are uniformly bounded in}\;\; L^{4/3}(0,T; L^{5/3}(\Omega))\cap L^{5/3}(0,T;L^{3/2}(\Omega))
\end{equation}
\begin{equation}\nonumber
\partial_t \sqrt{c_\e}=\frac12\left(\frac{\rho_\e}{\sqrt{c_\e}}-\sqrt{c_\e}+\frac{\Delta c_\e}{\sqrt{c_\e}}\right)  \;\text{are uniformly bounded in}\;\; L^{5/3}(\Omega\times(0,T))
\end{equation}
and
\begin{equation}\non
\partial_t\log c_\e=\frac{\rho_\e}{c_\e}-1+\Delta \log c_\e+\frac{|\nabla c_\e|^2}{c_\e^2} \;\text{are uniformly bounded in}\;\; L^{5/3}(\Omega\times (0,T)).
\end{equation}
Then, one may  extract a subsequence (without relabeling)  such that
\begin{equation}\nonumber
\rho_\e\longrightarrow \rho(x,t)  \;\text{weakly in}\;\; L^{5/4}(0,T; W^{1,5/4}(\Omega))\cap L^{5/3}(\Omega\times(0,T)),
\end{equation} and morevoer, due  to Aubin's lemma,
\begin{equation}\non
c_\e\longrightarrow c(x,t)\;\;\text{strongly in}\; L^2(0,T;H^1(\Omega))\; \text{and weakly in}\; L^{4/3}(0,T; H^2(\Omega))\cap L^{2}(0,T;W^{2,3/2}(\Omega)),
\end{equation}
\begin{equation}\non
\sqrt{c_\e}\longrightarrow h(x,t)\;\;\text{strongly in}\; L^2(0,T;H^1(\Omega))\cap C([0,T],L^4(\Omega))
\end{equation}and
\begin{equation}\non
\log c_\e\longrightarrow g(x,t) \;\;\text{strongly in}\; L^2(0,T;H^1(\Omega))\cap L^4(0,T;L^4(\Omega))
\end{equation}
and hence a.e. in $\Omega\times (0,T)$. Therefore, we conclude that $h=\sqrt{c}$ and $g=\log c$ by the uniqueness of limit and thanks to \cite[Corollary 4]{Simon87}, we have
\begin{equation}\non
c_\e\longrightarrow c(x,t)\;\;\text{strongly in}\; C([0,T],L^2(\Omega))\;\;\text{and}\; c(x,0)=c_I(x).
\end{equation}

Next, for any $\phi\in C_0^1(\Omega)$, it follows from the first equation that
\begin{align*}
	\left|\int_\Omega \partial_t\rho_\e \phi dx\right|\leq&\int_\Omega |\nabla \rho_\e||\nabla \phi|dx+\int_\Omega \rho_\e(1-\e\rho_\e)|\nabla \log c_\e||\nabla \phi|dx\non\\
	\leq&\|\nabla \rho_\e\|_{L^{5/4}(\Omega)}\|\nabla \phi\|_{L^\infty(\Omega)}+\left(\int_\Omega \frac{\rho_\e|\nabla c_\e|^2}{c_\e}dx \right)^{1/2}\left(\int_\Omega \frac{\rho_\e}{c_\e}\right)^{1/2}\|\nabla \phi\|_{L^\infty(\Omega)}\non\\
	\leq & \|\nabla \rho_\e\|_{L^{5/4}(\Omega)}\|\nabla \phi\|_{L^\infty(\Omega)}+\frac{1}{\sqrt{c_*}}\left(\int_\Omega \frac{\rho_\e|\nabla c_\e|^2}{c_\e}dx \right)^{1/2}\left(\int_\Omega \rho_\e\right)^{1/2}\|\nabla \phi\|_{L^\infty(\Omega)}\non\\
	\leq& C\left(\|\nabla \rho_\e\|_{L^{5/4}(\Omega)}+\left(\int_\Omega \frac{\rho_\e|\nabla c_\e|^2}{c_\e}dx \right)^{1/2}\right)\|\nabla \phi\|_{L^\infty(\Omega)}.
\end{align*}
Thus, we have 
\begin{equation}\non
\partial_t\rho_\e \;\text{are uniformly bounded in}\;\; L^{5/4}(0,T; (C_0^1(\Omega))^*)
\end{equation}and hence we may conclude by the Aubin lemma and the Ascoli lemma that (cf. \cite[Lemma 4.2]{CLM08})
\begin{equation*}
\rho_\e\longrightarrow \rho(x,t)\;\;\text{strongly in}\;C([0,T]; (C_0^1(\Omega))^*)\cap L^p(\Omega\times(0,T))\;\;\text{for any}\;p\in[1,\frac{5}{3})
\end{equation*} and hence a.e. in $\Omega\times(0,T)$. It follows that
\begin{equation*}
	\rho_\e(1-\e\rho_\e)\nabla \log c_\e\longrightarrow \rho \nabla \log c\;\;\text{a.e. in}\;\Omega\times(0,T).
\end{equation*}Recalling \eqref{r53} and \eqref{nalog4},  $\rho_\e(1-\e\rho_\e)\nabla \log c_\e$ are uniformly bounded in $L^{20/{17}}(\Omega\times(0,T))$. Thus, we deduce that
\begin{equation*}
\rho_\e(1-\e\rho_\e)\nabla \log c_\e\longrightarrow \rho \nabla \log c\;\;\text{weakly in}\;L^{20/{17}}(\Omega\times(0,T)).
\end{equation*}
Now, we may pass to the limit as $\e\rightarrow0$ to conclude that $(\rho,c)$ is a weak solution in the sense of Definition \ref{defwk}. This completes the proof for Theorem \ref{TH1}.\qed


\section{Eventual Regularity of the Weak Solutions}
In this section,  we study the eventual regularity and exponential stabilization of weak solutions in the three-dimensional case. Our method is based on a delicate analysis of the linearized semigroup and the stability of the spatially homogeneous solution in scaling-invariant spaces.
\subsection{The Linearized System}
In this section, we analyze the decay property of the linearized system around the spatially homogeneous solution $(\M,\M)$.   To this aim, we  introduce the reduced quantities $u_\e=\rho_\e-\M$ and $v_\e=c_\e-\M$ that satisfy
\begin{equation}
\begin{cases}\label{app2a}
u_{\e t}-\Delta u_\e=\nabla\cdot\bigg((\M+u_\e)\big(1-\e(u_\e+\M)\big)\nabla v_\e/(\M+v_\e)\bigg),\qquad &x\in\Omega,\;t>0\\
v_{\e t}-\Delta v_\e+v_\e=u_\e,\qquad &x\in\Omega,\;t>0\\
\frac{\partial u_\e}{\partial\nu}=\frac{\partial v_\e}{\partial\nu}=0,\qquad &x\in\partial\Omega,\;t>0\\
u_\e(x,0)=u_I(x)\triangleq\rho_I-\M,\;\; v_\e(x,0)= v_I(x)\triangleq c_I-\M,\qquad & x\in\Omega.
\end{cases}
\end{equation}
Since $\log(\M+z)=\log\M(1+z/\M)\approx\log\M+z/\M$ for $|z|\ll1$, the corresponding linearized system reads:
\begin{equation}
\begin{cases}\label{app2al}
\u_{\e t}-\Delta \u_\e-a \Delta \v_\e=0,\qquad &x\in\Omega,\;t>0\\
\v_{\e t}-\Delta \v_\e+\v_\e= u_\e,\qquad &x\in\Omega,\;t>0\\
\frac{\partial \u_\e}{\partial\nu}=\frac{\partial \v_\e}{\partial\nu}=0,\qquad &x\in\partial\Omega,\;t>0
\end{cases}
\end{equation}where $a=1-\e\M$. Here and below, we require that $0\leq \e\leq\frac{1}{2\M}$ such that $\frac{1}{2}\leq a\leq 1.$ Note that if $\e=0$, $a=1.$ From now on, we omit the subscript $\e$ since the results within this part are independent of $\e.$

Denote by $\Delta$ the usual Laplacian operator with homogeneous Neumann boundary condition and recall that $L^p_0(\Omega)$ is the Banach space of all functions $w\in L^p(\Omega)$ such that $\int_\Omega w dx=0$. Then $-\Delta$ is analytic on $L^2_0(\Omega)$ with domain $D_2(\Delta)=H^2_N(\Omega)\cap L^2_0(\Omega)$ and $I-\Delta$ is analytic on $L^2(\Omega)$ with domain $H^2_N(\Omega)$, respectively. Here $H^2_N(\Omega)\triangleq\{w\in H^2(\Omega):\;\partial_\nu w=0\;\text{on}\;\partial \Omega\}$. Furthermore, we can define the power $(-\Delta)^s$ as well as $(I-\Delta)^s$ for any $s\in\mathbb{R}$. We denote the domain of $(-\Delta)^s$ in $L^2_0(\Omega)$ by $D_2((-\Delta)^s)$ and the domain of $(I-\Delta)^s$ in $L^2(\Omega)$ by $D_2((I-\Delta)^s)$, respectively. Then it is well-known that $D_2((I-\Delta)^{1/2})=H^1(\Omega).$

Let $\mathcal{X}=L^2_0(\Omega)\times H^1(\Omega)$  with norm
\begin{equation}\nonumber
\|(u,v)\|_{\mathcal{X}}=\|u\|_{L^2(\Omega)}+\|(I-\Delta)^{1/2} v\|_{L^2(\Omega)}
\end{equation}
and define
\begin{align}\non
\mathcal{A}=\left(\begin{matrix}\Delta &a\Delta\\
1&\Delta -1\end{matrix}\right)
\end{align}with domain $D(\mathcal{A})=D_2(\Delta)\times  D_2((I-\Delta)^\frac32)$.
We observe that
\begin{align}
\mathcal{A}=\left(\begin{matrix}\Delta &0\\
0&\Delta -1\end{matrix}\right)+\left(\begin{matrix}0 &a\Delta\\
1&0\end{matrix}\right)
\triangleq \Lambda+\mathcal{U}\nonumber
\end{align}
where $\Lambda$ is a sectorial operator on $\mathcal{X}$. Moreover, one easily verifies that 
$D(\mathcal{A})=D(\Lambda)\subset D(\mathcal{U})=  H^1(\Omega)\times H^2_N(\Omega)$ and 
for any $(u,v)\in D(\Lambda)$, by interpolation, there holds
\begin{equation}
a\|\Delta v\|_{L^2}+\| u\|_{H^1}\leq \delta\bigg(\|\Delta u\|_{L^2}+\| \Delta v- v\|_{H^1}\bigg)+K_{\delta}(\|u\|_{L^2}+\| v\|_{H^1}).\nonumber
\end{equation} 
Then Lemma \ref{pertur1} indicates that $\mathcal{A}$ is a sectorial operator as well. 
For the sake of convenience, we denote 
\begin{align}
\left(\begin{matrix}\u(t)\\
\v(t)\end{matrix}\right)=e^{t\mathcal{A}}\left(\begin{matrix}u_I\\
v_I\end{matrix}\right)\triangleq\left(\begin{matrix}\Phi_1^t(u_I,v_I)\\
\Phi_2^t(u_I,v_I)\end{matrix}\right).\nonumber
\end{align}

After the above preparations, we recall the following result given in \cite[Lemma 4.2-Lemma 4.4]{J19}.
\begin{lemma}\label{lmexdecaypp}
	For any given initial data $u_I\in L^2_0(\Omega)$ and $v_I\in H^1(\Omega)$, the solution of \eqref{app2al} satisfies the following exponentially decay estimate
	\begin{equation}\nonumber
	\|\u\|^2+a\|\nabla \v\|^2\leq e^{-2\lambda_1t}(\|u_I\|^2+a\|\nabla v_I\|^2)\qquad\text{for all}\;t\geq0.
	\end{equation}
\end{lemma}

\begin{lemma}\label{keylem1} Assume $d\geq2$. Then for any $u_I\in C(\overline{\Omega})\cap L^1_0(\Omega)$ and $v_I\in C^1(\overline{\Omega})$ satisfying  $\partial_\nu v_I=0$ on $\partial\Omega$,  there hold
	\begin{equation}
	\|\tilde{u}(t)\|_{L^p(\Omega)}\leq Ce^{-\lambda_1t}(1+t^{-\frac{d}{2}(\frac2d-\frac1p)})\big(\|u_I\|_{L^{d/2}(\Omega)}+\|\nabla v_I\|_{L^d(\Omega)}\big)\qquad\forall t>0
	\end{equation} for any $p>1$ satisfying  $\frac d2\leq p<\infty$,  and
	\begin{equation}
	\|\nabla \tilde{v}(t)\|_{L^p(\Omega)}\leq Cpe^{-\lambda_1t}(1+t^{-\frac{d}{2}(\frac1d-\frac1p)})\big(\|u_I\|_{L^{d/2}(\Omega)}+\|\nabla v_I\|_{L^d(\Omega)}\big)\qquad\forall t>0
	\end{equation}for any $p$ satisfying $d\leq p< \infty$, where $C>0$ depends on $d$, $\M$ and $\Omega$ only.
\end{lemma}
\begin{remark}\label{remark}
	Under the assumption of Lemma \ref{keylem1}, for any $2\leq l\leq p<\infty$,  there holds	\begin{equation}\label{lprem}
	\|\u(t)\|_{L^p(\Omega)}+\|\nabla\v(t)\|_{L^p(\Omega)}\leq Ce^{-\lambda_1t}(1+t^{-\frac d2(\frac1l-\frac1p)})(\|u_I\|_{L^{l}(\Omega)}+\|\nabla v_I\|_{L^l(\Omega)})
	\end{equation}with $C$ depends on $d$, $\M$ and $\Omega$ only.
\end{remark}
\begin{lemma}
	\label{Cor}Assume $d\geq2$. Suppose $u_I=\nabla\cdot w_I$ and $v_I=0$. Then there holds
	\begin{equation}
	\|\tilde{u}(t)\|_{L^p(\Omega)}\leq Ce^{-\lambda_1t}(1+t^{-\frac12-\frac{d}{2}(\frac1q-\frac1p)})\|w_I\|_{L^{q}(\Omega)}
	\end{equation} for any $\frac d2< q\leq p<\infty$, where $C>0$ depends on $d$, $\M$ and $\Omega$  if $d\geq3$ and also depends on $1/|q-1|$ if $d=2$. 	
\end{lemma}

\subsection{Stabilization in Scaling-invariant Spaces}
Now, we are ready to prove the following stability result for system \eqref{app2a} based on a one-step contraction argument (cf. \cite[Lemma 5.1]{J18b}).
\begin{proposition}\label{propB}
	Assume that $d\geq2$ and $\M>0$. For any fixed $d<q_0<2d$, $q_0< p_0<\frac{dq_0}{q_0-d}$ and $0<\mu'<\lambda_1,$ there exists $\eta_0>0$ which may depend on $d$, $q_0,p_0,\mu'$, $\M$ and $\Omega$ but is independent of $\e$ such that for any initial data $(u_I,v_I)\in C(\overline{\Omega})\cap L^1_0(\Omega)\times C^1(\overline{\Omega})$ satisfying  $\partial_\nu v_I=0$ on $\partial\Omega$,  $u_I\geq-\M$, $v_I>-\M$ in $\overline{\Omega}$ and $\|u_I\|_{L^{d/2}(\Omega)}+\|\nabla v_I\|_{L^d(\Omega)}+\|v_I\|_{L^\infty(\Omega)}\leq \eta$ for any $\eta<\eta_0$, system \eqref{app2a} with $\e\in(0,\tilde{\zeta}_2]$ has a global classical solution $(u_\e,v_\e)$ which is bounded such that for all $t\geq0,$
\begin{equation}\nonumber
\|u_\e(t)-\u_\e(t)\|_{L^{q_0}(\Omega)}\leq\eta e^{-\mu't}(1+t^{-1+\frac{d}{2q_0}})
\end{equation}where $\tilde{\zeta}_2\triangleq\min\{1/{2\M},\tilde{\zeta}_1\}$ with $1/{\tilde{\zeta}_1}\triangleq\|u_I+\M\|_{L^\infty}$ and $(\u_\e,\v_\e)$ is the solutions of the linearized system \eqref{app2al} with the same initial data.
\end{proposition}
\begin{proof}According to Theorem \ref{locala}, problem \eqref{chemo1a} with  initial data $\rho_I=u_I+\M$ and $c_I=v_I+\M$ has a unique classical solution on $[0,t_\e^+)$ and if  $t_\e^+<\infty,$ we have \beq\limsup\limits_{t\nearrow t_\e^+}\|\rho_\e(\cdot,t)\|_{L^\infty}=\infty.\non\eeq  Accordingly, for problem \eqref{app2a}, we obtain a unique classical solution $(u_\e,v_\e)=(\rho_\e-\M,c_\e-\M)$ on $[0,t_\e^+)$. Note that $t_\e^+=\infty$ if $\e\in(0,\tilde{\zeta}_1]$ and there holds $\|\rho_\e(\cdot,t)\|_{L^\infty}\leq 1/{\e}.$ In addition, thanks to the uniform-in-time lower bound for $c_\e(t)$ in Lemma \ref{unilow}, we also have for all $t\in[0,t_\e^+
)$,
	\begin{equation}\label{unilow2}
		\inf\limits_{x\in\Omega}v_\e(x,t)+\M\geq c_*>0
	\end{equation}
with $c_*$ depending only on $\M$ and $\Omega$ since we may require that $\eta_0<\min\{1,\frac{\M}{2}\}$.

	Now, for any fixed $\e\in(0,\tilde{\zeta}_2]$ we define
	\begin{align}\label{defTa}
	T_\e\triangleq\sup\left.\bigg{\{}T>0:\|u_\e(t)-\u_\e(t)\|_{L^{q_0}(\Omega)}\leq\eta e^{-\mu't}(1+t^{-1+\frac{d}{2q_0}}),\text{for all}\;t\in[0,T).\bigg{\}}\right.
	\end{align}
	Then $T_\e$ is well-defined and positive with $T_\e<t_\e^+$. Indeed, near $t=0$, $\|u_\e(t)\|_{L^\infty(\Omega)}$ is bounded due to Theorem \ref{locala} and $\|\u_\e(t)\|_{L^{q_0}(\Omega)}$ is uniformly bounded due to Remark \ref{remark} (taking $p=l=q_0$), while on the other hand, as $t\rightarrow0^+,$ $ t^{-1+\frac{d}{2q_0}}\rightarrow+\infty$.
	
	Now we claim that if $\eta_0$ is chosen sufficiently small, one has $T_\e=\infty.$
	First, we infer from Lemma \ref{keylem1} and Definition \ref{defTa} that 
	\begin{equation}\non
		\|u_\e(t)\|_{L^{q_0}(\Omega)}\leq C\eta e^{-\mu't}(1+t^{-1+\frac{d}{2q_0}})
	\end{equation}with $C$ depending on $d$, $\M$ and $\Omega$ only.
	On the other hand, it follows from the variation-of-constant formula that 
	\begin{align}
	v_\e(t)&=e^{t(\Delta-1)}v_I+\int_0^t e^{(t-s)(\Delta-1)}u_\e(s)ds\non\\
	&=e^{t(\Delta-1)}v_I+\int_0^t e^{(t-s)(\Delta-1)}\u_\e(s)ds+\int_0^t e^{(t-s)(\Delta-1)}(u_\e(s)-\u_\e(s))ds\non\\
	&=\v_\e(t)+\int_0^t e^{(t-s)(\Delta-1)}(u_\e(s)-\u_\e(s))ds.\nonumber
	\end{align}
	Applying $\nabla$ to both sides of the above identity, thanks to  Lemma \ref{lmint} and Lemma \ref{lmpq}, we infer that
	\begin{align}
	\|\nabla v_\e(t)-\nabla \tilde{v}_\e(t)\|_{L^{p_0}(\Omega)}\leq& \int_0^t\|\nabla e^{(t-s)(\Delta-1)}(u_\e(s)-\tilde{u}_\e(s))\|_{L^{p_0}(\Omega)}ds\nonumber\\
	\leq&C\int_0^t (1+(t-s)^{-\frac12-\frac{d}{2}(\frac1{q_0}-\frac1{p_0})})e^{-(\lambda_1+1)(t-s)}\|u_\e(s)-\tilde{u}_\e(s)\|_{L^{q_0}(\Omega)}ds\nonumber\\
	\leq& C\int_0^t (1+(t-s)^{-\frac12-\frac{d}{2}(\frac1{q_0}-\frac1{p_0})})e^{-(\lambda_1+1)(t-s)}\eta e^{-\mu's}(1+s^{-1+\frac{d}{2q_0}})ds\nonumber\\
	\leq& C\eta (1+t^{-\frac{1}{2}+\frac{d}{2p_0}})e^{-\mu't}\nonumber
	\end{align}
	which together with Lemma \ref{keylem1} implies that 
	\begin{align}
	\|\nabla v_\e(t)\|_{L^{p_0}(\Omega)}\leq& \|\nabla v_\e(t)-\nabla \tilde{v}_\e(t)\|_{L^{p_0}(\Omega)}+\|\nabla \tilde{v}_\e(t)\|_{L^{p_0}(\Omega)}\non\\
	\leq&C\eta (1+t^{-\frac{1}{2}+\frac{d}{2p_0}})e^{-\mu't}\nonumber
	\end{align}for all $t\in[0,T_\e)$. In addition, by Lemma \ref{lmpq}-(i), there holds that (note $v_I\notin L^1_0(\Omega)$ and $e^{t\Delta}a=a$ for any constant $a$)
	\begin{align}
	\|v_\e(t)\|_{L^\infty(\Omega)}\leq &\|e^{t(\Delta-1)}v_I\|_{L^\infty(\Omega)}+\int_0^t \|e^{(t-s)(\Delta-1)}u_\e(s)\|_{L^\infty(\Omega)}ds\non\\
	\leq&k_1e^{-(\lambda_1+1)t}\|v_I-\overline{v_I}\|_{L^\infty(\Omega)}+|\overline{v_I}|e^{-t}\non\\
	&\qquad+C\int_\Omega e^{-(\lambda_1+1)(t-s)}(1+(t-s)^{-\frac{d}{2q_0}})\eta e^{-\mu's}(1+s^{-1+\frac{d}{2q_0}})ds\non\\
	\leq& C\eta e^{-t}+C\eta e^{-\mu't}
	\end{align}	where $C>0$ depends on $\Omega$, $q_0$ and $d$ only.

Next, exploiting the semigroup $e^{t\mathcal{A}}$ and variation-of-constants formula again, we infer that  for all $0<t<t_\e^+,$
	\begin{align}\label{uvexp}
	&\left(\begin{matrix}u_\e(t)\\
	v_\e(t)\end{matrix}\right)\non\\
	=&e^{t\mathcal{A}}	\left(\begin{matrix}u_I\\
	v_I\end{matrix}\right)+\int_0^te^{(t-s)\mathcal{A}}	\left(\begin{matrix}\nabla\cdot\bigg(\left(u_\e(1-2\e\M-\e u_\e)-(1-\e\M)v_\e\right)\frac{\nabla v_\e}{v_\e+\M}\bigg)\\
	0\end{matrix}\right)ds
	\end{align}
	from which we represent $u_\e$ according to
	\begin{align}\label{uform}
	u_\e(t)=\u_\e(t)+\int_0^t\Phi_1^{t-s}\left(\nabla\cdot\bigg(\Big((1-\e\M-\e u_\e)-\e\M  \Big)u_\e-(1-\e\M)v_\e\bigg)\frac{\nabla v_\e}{v_\e+\M},0\right)ds.
	\end{align}
	Denote by $\frac{1}{r_0}=\frac{1}{q_0}+\frac{1}{p_0}$. Thus due to Lemma \ref{Cor}, Lemma \ref{lmint} and the point-wise estimates $0\leq (1-\e u_\e-\e\M)\leq 1$ and $v_\e+\M\geq c_*$, we deduce that
	\begin{align}\label{mp1a}
	&\|u_\e(t)-\u_\e(t)\|_{L^{q_0}(\Omega)}\non\\
	\leq&\int_0^t\bigg\|\Phi_1^{t-s}\left(\nabla\cdot\bigg(\Big((1-\e\M-\e u_\e)-\e\M  \Big)u_\e-(1-\e\M)v_\e\bigg)\frac{\nabla v_\e}{v_\e+\M},0\right) \bigg\|_{L^{q_0}(\Omega)}ds\non\\
	\leq&C\int_0^t e^{-\lambda_1(t-s)}(1+(t-s)^{-\frac12-\frac{d}{2p_0}})\non\\
	&\qquad\times\left\|\bigg(\Big((1-\e\M-\e u_\e)-\e\M  \Big)u_\e-(1-\e\M)v_\e\bigg)\frac{\nabla v_\e}{v_\e+\M}\right\|_{L^{r_0}(\Omega)}ds\nonumber\\
	\leq&C\int_0^t e^{-\lambda_1(t-s)}(1+(t-s)^{-\frac12-\frac{d}{2p_0}})\left(\|u_\e\nabla v_\e\|_{L^{r_0}(\Omega)}+\|v_\e\nabla v_\e\|_{L^{r_0}(\Omega)}\right)ds\nonumber\\
	\leq &C\int_0^t e^{-\lambda_1(t-s)}(1+(t-s)^{-\frac12-\frac{d}{2p_0}})\left(\|u_\e\|_{L^{q_0}(\Omega)}\|\nabla v_\e\|_{L^{p_0}(\Omega)}+\|v_\e(s)\|_{L^\infty(\Omega)}\|\nabla v_\e\|_{L^{p_0}(\Omega)}\right)ds\nonumber\\
	\leq&C\eta^2\int_0^t e^{-\lambda_1(t-s)}e^{-2\mu's}(1+(t-s)^{-\frac12-\frac{d}{2p_0}})(1+s^{-\frac32+\frac{d}{2r_0}})ds\nonumber\\
	&\quad+ C\eta^2\int_0^t e^{-\lambda_1(t-s)}(e^{-(\mu'+1)s}+e^{-2\mu's})(1+(t-s)^{-\frac12-\frac{d}{2p_0}})(1+s^{-\frac12+\frac{d}{2p_0}})ds\non\\
	\leq&\tilde{C}\eta^2e^{-\mu't}(1+t^{-1+\frac{d}{2q_0}}),\end{align}
	where $\tilde{C}$ may depend on  $p_0, q_0, d,\mu'$, $\M$ and $\Omega$, but is independent of $\varepsilon,t$, $T_\e$, $\eta$.
	Then taking $\tilde{C}\eta_0<\frac12$, we conclude that $T_\e=\infty$. This completes the proof.	
\end{proof}
\noindent\textbf{Proof of Theorem \ref{TH3}.}\\
	Theorem \ref{TH3}  is a direct consequence of Proposition \ref{propB} with $\e=0$. Since the proof  is the same as those in \cite{J18a,J18b} we omit the detail here. \qed

\subsection{Proof of Theorem \ref{TH2}}
In order to apply Proposition \ref{propB} to prove Theorem \ref{TH2}, we need the following auxiliary lemma.

\begin{lemma}\label{decay00}
	For $d=3$ and all $t>0$, we have
	\begin{equation}\label{decay00a}
	\|\rho_\e(t)-\M\|^2_{L^{3/2}(\Omega)}\leq C\int_\Omega\frac{|\nabla \rho_\e|^2}{\rho_\e} dx
	\end{equation} and
	\begin{equation}\label{decay00b}
	\|\nabla c_\e(t)\|^2_{L^3(\Omega)}+\|c_\e(t)-\overline{c_\e}(t)\|^2_{L^\infty(\Omega)}\leq C\int_\Omega\left( c_\e|\nabla^2\log c_\e|^2+ \frac{|\nabla  c_\e|^2}{c_\e}\right)dx
	\end{equation}where $C$ depends on $\Omega$ and the initial data.
	
	In addition, if $\Omega$ is convex and $\e\in(0,\zeta_0]$, there holds
	\begin{equation}\label{timeint0}
	\int_0^\infty\left(	\|\rho_\e(t)-\M\|^2_{L^{3/2}(\Omega)}+\|\nabla c_\e(t)\|^2_{L^3(\Omega)}+\|c_\e(t)-\M\|^2_{L^\infty(\Omega)}\right)dt\leq C
	\end{equation}where $C$ depends on $\Omega$ and the initial data.
\end{lemma}
\begin{proof}
	Observing that $\int_\Omega \rho_\e dx$ is conserved, one easily deduce by H\"older's inequality that
	\begin{equation}\label{idea1}
	\left(\int_\Omega |\nabla \rho_\e|dx\right)^2\leq\left(\int_\Omega\frac{|\nabla \rho_\e|^2}{\rho_\e}dx\right)\left(\int_\Omega\rho_\e dx\right),
	\end{equation}and in the three-dimensional setting, the critical continuous embedding $W^{1,1}(\Omega)\hookrightarrow L^{3/2}(\Omega)$ together with a Poincar\'e-Sobolev inequality  yields that
	\begin{equation}\label{idea2}
	\|\rho_\e-\M\|_{L^{3/2}(\Omega)}\leq C\|\nabla \rho_\e\|_{L^1(\Omega)}.
	\end{equation} 	
	Then \eqref{decay00a} follows from \eqref{idea1} and \eqref{idea2}.

	Recalling that $\|c_\e\|_{L^1(\Omega)}=e^{-t}\|c_I\|_{L^1(\Omega)}+(1-e^{-t})\|\rho_I\|_{L^1(\Omega)}$, then by the three-dimensional Agmon inequality and the Poincar\'e inequality we deduce that
	\begin{align}
	\|\sqrt{c_\e}\|_{L^\infty}^2\leq&\|\sqrt{c_\e}\|_{H^2}\|\sqrt{c_\e}\|_{H^1}\non\\
	\leq&C(\|\Delta\sqrt{c_\e}\|+\|\sqrt{c_\e}\|)(\|\nabla \sqrt{c_\e}\|+\|\sqrt{c_\e}\|)\non\\
	\leq& C\|\Delta\sqrt{c_\e}\|\|\nabla \sqrt{c_\e}\|+C\|\sqrt{c_\e}\|\|\Delta\sqrt{c_\e}\|+C\|\sqrt{c_\e}\|^2\non\\
	\leq&C\|\Delta\sqrt{c_\e}\|\|\nabla \sqrt{c_\e}\|+C\|\Delta\sqrt{c_\e}\|+C.\nonumber
	\end{align}
	Therefore, since $\|\nabla \sqrt{c_\e}\|$ is bounded for $t\geq0$, by interpolation and the three-dimensional Sobolev embedding together with the Young inequality, we infer that (note that $\mathrm{curl}\nabla \phi=0$, $\forall \phi$)
	\begin{align}
	\|\nabla c_\e\|_{L^3(\Omega)}^3=& 8\int_\Omega|\nabla \sqrt{c_\e}|^3 c_\e^{3/2}dx\non\\
	\leq&C\|c_\e\|^{3/2}_{L^\infty(\Omega)}\|\nabla \sqrt{c_\e}\|^3_{L^3(\Omega)}\non\\
	\leq&C\|c_\e\|^{3/2}_{L^\infty(\Omega)}\|\nabla \sqrt{c_\e}\|_{L^6(\Omega)}^{3/2}\|\nabla \sqrt{c_\e}\|^{3/2}\non\\
	\leq&C\|c_\e\|^{3/2}_{L^\infty(\Omega)}\|\Delta \sqrt{c_\e}\|^{3/2}\|\nabla \sqrt{c_\e}\|^{3/2}\non\\
	\leq&C\left(\|\Delta\sqrt{c_\e}\|^{3/2}\|\nabla \sqrt{c_\e}\|^{3/2}+\|\Delta\sqrt{c_\e}\|^{3/2}+1\right)\|\Delta \sqrt{c_\e}\|^{3/2}\|\nabla \sqrt{c_\e}\|^{3/2}\non\\
	\leq& C\|\Delta \sqrt{c_\e}\|^{3}\|\nabla \sqrt{c_\e}\|^{3}+ C\|\Delta \sqrt{c_\e}\|^{3}\|\nabla \sqrt{c_\e}\|^{3/2}+C\|\Delta \sqrt{c_\e}\|^{3/2}\|\nabla \sqrt{c_\e}\|^{3/2}\non\\
	\leq&C\|\Delta \sqrt{c_\e}\|^{3}+C\|\nabla \sqrt{c_\e}\|^{3}\nonumber
	\end{align}
	which together with Lemma \ref{lm6} entails that
	\begin{equation}
	\|\nabla c_\e\|^2_{L^3(\Omega)}\leq C\|\Delta \sqrt{c_\e}\|^2+C\|\nabla \sqrt{c_\e}\|^2\leq C\int_\Omega\left( c_\e|\nabla^2\log c_\e|^2+\frac{|\nabla c_\e|^2}{c_\e}\right)dx.\nonumber
	\end{equation}
	On the other hand, by Lemma \ref{lm6}, the three-dimensional Gagliardo-Nireberg inequality, Poincar\'e's inequality and the Young inequality, we obtain that
	\begin{align}
	\|c_\e-\overline{c_\e}\|^4_{L^\infty(\Omega)}\leq&C\|\nabla c_\e\|_{L^6(\Omega)}^2\|c-\overline{c_\e}\|_{L^6(\Omega)}^2 \non\\
	\leq&C\| \Delta c_\e\|^2\| \nabla c_\e\|^2\non\\
	\leq& C\|c_\e\|^2_{L^\infty(\Omega)}\left(\int_\Omega \frac{|\Delta c_\e|^2}{c_\e}dx\right)\left(\int_\Omega \frac{|\nabla  c_\e|^2}{c_\e}dx\right)\non\\
	\leq&C(\|\Delta\sqrt{c_\e}\|^2\|\nabla \sqrt{c_\e}\|^2+\|\Delta\sqrt{c_\e}\|^2+1)\left(\int_\Omega \frac{|\Delta c_\e|^2}{c_\e}dx\right)\left(\int_\Omega \frac{|\nabla  c_\e|^2}{c_\e}dx\right)\non\\
	\leq&C\left(\int_\Omega c_\e|\nabla^2\log c_\e|^2dx\right)^2+C\left(\int_\Omega \frac{|\nabla  c_\e|^2}{c_\e}dx\right)^2.\nonumber
	\end{align}
	Then it follows that
	\begin{equation}
	\|c_\e-\overline{c_\e}\|^2_{L^\infty(\Omega)}\leq C\int_\Omega\left(  c_\e|\nabla^2\log c_\e|^2+ \frac{|\nabla  c_\e|^2}{c_\e}\right)dx\nonumber
	\end{equation}and hence
	\begin{align}
	\|c_\e(t)-\M\|^2_{L^\infty(\Omega)}\leq& 2\|c_\e-\overline{c_\e}\|^2_{L^\infty(\Omega)}+2|\overline{c_\e}(t)-\M|^2\non\\
	\leq& C\int_\Omega\left(  c_\e|\nabla^2\log c_\e|^2+ \frac{|\nabla  c_\e|^2}{c_\e}\right)dx	+2|\overline{c_I}-\M|e^{-2t}.\nonumber
	\end{align}
	As a result, we infer that
	\begin{align}
	&	\int_0^\infty\left(	\|\rho_\e(t)-\M\|^2_{L^{3/2}(\Omega)}+\|\nabla c_\e(t)\|^2_{L^3(\Omega)}+\|c_\e(t)-\M\|^2_{L^\infty(\Omega)}\right)dt\non\\
	\leq&C\int_0^\infty\int_\Omega \left(\frac{|\nabla \rho_\e|^2}{\rho_\e} +c_\e|\nabla^2\log c_\e|^2+\frac{|\nabla  c_\e|^2}{c_\e}\right)dxdt+2|\overline{c_I}-\M|\int_0^\infty e^{-2t}dt\non\\
	\leq &C\nonumber
	\end{align}due to Lemma \ref{lya}. This completes the proof.
\end{proof}

\noindent\textbf{Proof of Theorem \ref{TH2}}.\\
With above preparation, the following proof is basically the same as those in \cite{J19}. Since our method is not conventional, we report argument in detail  for reader's convenience.

For any given initial data $\rho_I,c_I$ satisfying \eqref{ini}, $\M=\frac{1}{|\Omega|}\int_\Omega \rho_Idx$, $1/\zeta_0=\|\rho_I\|_{L^\infty}$ and  $\zeta_2'=\min\{1/{2\M},\zeta_0\}$ are all well-defined.  Keep in mind that $\rho_\e=u_\e+\M$ and $c_\e=v_\e+\M$.  Since now $d=3$, we can fix $3<q_0<6$, $q_0< p_0<\frac{3q_0}{q_0-3}$ and $0<\mu'<\lambda_1$. Then we can fix $\eta_0>0$ according to Proposition \ref{propB} which is independent of $\e$ and moreover we have the following result.
\begin{lemma}
	For any $\eta<\eta_0$, there is $t_\eta>0$ which is independent of $\e$ such that for all $\e\in(0,\zeta'_2]$ there holds
	\begin{equation*}
	\|\rho_\e(t_\eta)-\M\|_{L^{3/2}(\Omega)}+\|\nabla c_\e(t_\eta)\|_{L^3(\Omega)}+\|c_\e(t_\eta)-\M\|_{L^\infty(\Omega)}\leq \eta.	
	\end{equation*}
\end{lemma}
\begin{proof}
Due to \eqref{timeint0},  for any $\eta<\eta_0,$ thanks to Lemma \ref{lm0},   we can always find a constant $\tilde{t}_\eta>0$ independent of $\e$ and  a time $\tilde{t}_{\e}(\eta)\in(0,\tilde{t}_\eta)$ such that
\begin{equation*}
\|\rho_\e(\tilde{t}_{\e}(\eta))-\M\|_{L^{3/2}(\Omega)}+\|\nabla c_\e(\tilde{t}_{\e}(\eta))\|_{L^3(\Omega)}+\|c_\e(\tilde{t}_{\e}(\eta))-\M\|_{L^\infty(\Omega)}\leq \eta<\eta_0.
\end{equation*}
Regard the solution for $t\geq \tilde{t}_{\e}(\eta)$ as a trajectory originating from $(\rho_\e(\tilde{t}_{\e}(\eta)), c_\e(\tilde{t}_{\e}(\eta)))$.  Since $\e\in(0,\zeta_2']$, we have $\|\rho_\e(\tilde{t}_{\e}(\eta))\|_{L^\infty(\Omega)}\leq 1/\e$ by Theorem \ref{locala}. Thus, by definition of $\zeta_2$  appearing in Proposition \ref{propB}, we have $1/\zeta_2=\|\rho_\e(\tilde{t}_{\e}(\eta))\|_{L^\infty(\Omega)}\leq1/\e$, which indicates $\e\in(0,\zeta_2]$. Now, we can apply Proposition \ref{propB} to deduce that for $t\geq \tilde{t}_{\e}(\eta)$, there holds
\begin{equation}\nonumber
\|u_\e(t)-\u_\e(t)\|_{L^{q_0}(\Omega)}\leq\eta e^{-\mu'(t-\tilde{t}_{\e}(\eta))}(1+(t-\tilde{t}_{\e}(\eta))^{-1+\frac{3}{2q_0}}).
\end{equation}
As a result, for all $t\geq \tilde{t}_{\e}(\eta)$, applying Lemma \ref{keylem1}, we have
\begin{align}\label{help5}
&\|\rho_\e(t)-\M\|_{L^{q_0}(\Omega)}=\|u_\e(t)\|_{L^{q_0}(\Omega)}\non\\
\leq&\|u_\e(t)-\u_\e(t)\|_{L^{q_0}(\Omega)}+\|\u_\e(t)\|_{L^{q_0}(\Omega)}\non\\
\leq&\eta e^{-\mu'(t-\tilde{t}_{\e}(\eta))}(1+(t-\tilde{t}_{\e}(\eta))^{-1+\frac{3}{2q_0}})+Ce^{-\lambda_1(t-\tilde{t}_{\e}(\eta))}(1+(t-\tilde{t}_{\e}(\eta))^{-1+\frac{3}{2q_0}})\|u_\e(\tilde{t}_{\e}(\eta))\|_{L^{3/2}(\Omega)}\non\\
&\qquad+Ce^{-\lambda_1(t-\tilde{t}_{\e}(\eta))}(1+(t-\tilde{t}_{\e}(\eta))^{-1+\frac{3}{2q_0}})\|\nabla v_\e(\tilde{t}_{\e}(\eta))\|_{L^3(\Omega)}\non\\
\leq& C \eta e^{-\mu'(t-\tilde{t}_{\e}(\eta))}(1+(t-\tilde{t}_{\e}(\eta))^{-1+\frac{3}{2q_0}})
\end{align}where $C>0$ is independent of $\e$ and $\eta$. In the same way as before, we deduce that
\begin{align}\label{help6}
&\|\nabla c_\e(t)\|_{L^{p_0}(\Omega)}\leq \|\nabla v_\e(t)-\nabla \tilde{v}_\e(t)\|_{L^{p_0}(\Omega)}+\|\nabla \v_\e(t)\|_{L^{p_0}(\Omega)}\non\\
&\leq C\eta e^{-\mu'(t-\tilde{t}_{\e}(\eta))}(1+(t-\tilde{t}_{\e}(\eta))^{-\frac12+\frac{3}{2p_0}})\qquad\text{for all}\;t\geq \tilde{t}_{\e}(\eta).
\end{align}	
Recall that $\tilde{t}_{\e}(\eta)<\tilde{t}_\eta$ and $\tilde{t}_\eta$ is independent of $\e$. We infer that for all $t\geq \tilde{t}_\eta+1$, there holds
\begin{equation}
\|\rho_\e(t)-\M\|_{L^{q_0}(\Omega)}+\|\nabla c_\e(t)\|_{L^{p_0}(\Omega)}\leq C\eta e^{-\mu't}
\end{equation}with $C>0$ independent of $\e$ and $\eta$. Then our assertion follows apparently if we pick $t_\eta>\tilde{t}_\eta+1$ sufficiently large.
\end{proof}

Once again, by uniqueness of classical solutions, we can regard the family of solutions $(\rho_\e,c_\e)$ for $t\geq t_\eta$ as a family of trajectories {\it uniformly} starting initially from $(\rho_\e(t_\eta), c_\e(t_\eta))$. Repeating the above argument and applying Proposition \ref{propB} yields that for $t\geq t_\eta$,
\begin{equation}\nonumber
\|u_\e(t)-\u_\e(t)\|_{L^{q_0}(\Omega)}\leq\eta e^{-\mu'(t-t_\eta)}(1+(t-t_\eta)^{-1+\frac{3}{2q_0}}),
\end{equation}
\begin{align}\label{help5b}
\|\rho_\e(t)-\M\|_{L^{q_0}(\Omega)}
\leq C \eta e^{-\mu'(t-t_\eta)}(1+(t-t_\eta)^{-1+\frac{3}{2q_0}}),
\end{align}and
\begin{align}\label{help6b}
\|\nabla c_\e(t)\|_{L^{p_0}(\Omega)}\leq C\eta e^{-\mu'(t-t_\eta)}(1+(t-t_\eta)^{-\frac12+\frac{3}{2p_0}}).
\end{align}	
In order to prove the eventual smoothness of the limiting functions, we still need  to derive certain higher-order estimates for $(\rho_\e,c_\e)$ that are independent of $\e$ (cf. \cite{Lankeit,Win19}).
\begin{lemma}\label{lm46}
	There is $C>0$ independent of $\e$ such that for all $t\geq t_\eta+1$,
	\begin{equation*}
		\|\rho_\e(t)-\M\|_{L^{\infty}(\Omega)}+\|\nabla c_\e(t)\|_{L^{\infty}(\Omega)}\leq C e^{-\mu't}.
	\end{equation*}
Moreover, there is $\theta\in(0,1)$ and $C>0$ independent of $\e$ such that for all $t\geq t_\eta+3$,
\begin{equation*}
	\|\rho_\e(t)\|_{C^{2+\theta,1+\frac{\theta}{2}}(\overline{\Omega}\times[t,t+1])}+\|c_\e(t)\|_{C^{2+\theta,1+\frac{\theta}{2}}(\overline{\Omega}\times[t,t+1])}\leq C.
\end{equation*}
\end{lemma}
\begin{proof}
	For the sake of simplicity we assume $t_\eta=0$. Then, by Lemma \ref{lmpq}-(ii) and \eqref{help5b}
	\begin{align}\label{vestimate0}
	&\|\nabla c_\e(t)\|_{L^\infty(\Omega)}=\|\nabla v_\e(t)\|_{L^\infty(\Omega)}\non\\
	\leq&\|\nabla e^{t(\Delta-1)}v_I\|_{L^\infty(\Omega)}+\int_0^t\|\nabla e^{(\Delta-1)(t-s)}u_\e(s)\|_{L^\infty(\Omega)}ds\nonumber\\
	\leq & Ce^{-(\lambda_1+1)t}(1+t^{-\frac12})\|v_I\|_{L^\infty(\Omega)}+C\int_0^t(1+(t-s)^{-\frac12-\frac{3}{2q_0}})e^{-(\lambda_1+1)(t-s)}\|u_\e(s)\|_{L^{q_0}(\Omega)}ds\nonumber\\
	\leq&Ce^{-(\lambda_1+1)t}(1+t^{-\frac12})+C\int_0^t(1+(t-s)^{-\frac12-\frac{3}{2q_0}})e^{-(\lambda_1+1)(t-s)}e^{-\mu's}(1+s^{-1+\frac{3}{2q_0}})ds\nonumber\\
	\leq& Ce^{-(\lambda_1+1)t}(1+t^{-\frac12})+Ce^{-\mu't}(1+t^{-\frac12})\non\\
	\leq& Ce^{-\mu't}(1+t^{-\frac12}).
	\end{align}which indicates when $t\geq1,$
	\begin{equation}\label{vestimate1a}
	\|\nabla c_\e(t)\|_{L^\infty(\Omega)}\leq Ce^{-\mu't}.
	\end{equation} In addition, thanks to \eqref{help5b}-\eqref{help6b}, the fact $0\leq1-\e(u_\e+\M)\leq1$ and the strictly positive uniform-in-time and uniform-in-$\e$ lower bounds of $c_\e$, we infer that when $t\geq1$
	 \begin{align}\label{uestimate0}
	&\|\rho_\e(t)-\M\|_{L^{\infty}(\Omega)}=\|u_\e(t)\|_{L^\infty(\Omega)}\non\\
	\leq& \|e^{(t-1)\Delta}u_\e(1)\|_{L^\infty(\Omega)}+\M\int_1^t\bigg\|e^{\Delta(t-s)}\nabla\cdot\left(\left(1-\e(u_\e+\M)\right) \frac{\nabla v_\e}{c_\e}\right)\bigg\|_{L^\infty(\Omega)}ds\non\\
	&+\int_1^t\left\|e^{\Delta(t-s)}\nabla\cdot\left(\left(1-\e(u_\e+\M)\right) \frac{u_\e\nabla v_\e}{c_\e}\right)\right\|_{L^\infty(\Omega)}ds\nonumber\\
	\leq& k_1e^{-\lambda_1(t-1)}(1+t^{-\frac{3}{2q_0}})\|u_\e(1)\|_{L^{q_0}(\Omega)}+C\int_1^t(1+(t-s)^{-\frac12-\frac{3}{2p_0}})e^{-\lambda_1(t-s)}\|\nabla v_\e(s)\|_{L^{p_0}(\Omega)}ds\nonumber\\
	&+C\int_1^t(1+(t-s)^{-\frac12-\frac{3}{2q_0}})e^{-\lambda_1(t-s)}\|u_\e(s)\|_{L^{q_0}(\Omega)}\|\nabla v_\e(s)\|_{L^{\infty}(\Omega)}ds\nonumber\\
	\leq&Ce^{-\lambda_1t}+C\int_1^t(1+(t-s)^{-\frac12-\frac{3}{2p_0}})e^{-\lambda_1(t-s)}(1+s^{-\frac12+\frac{3}{2p_0}})e^{-\mu's}ds\nonumber\\
	&+C\int_1^t(1+(t-s)^{-\frac12-\frac{3}{2q_0}})e^{-\lambda_1(t-s)}(1+s^{-1+\frac{3}{2q_0}})e^{-2\mu's}ds\nonumber\\
	\leq&Ce^{-\lambda_1t}+C\int_0^t(1+(t-s)^{-\frac12-\frac{3}{2p_0}})e^{-\lambda_1(t-s)}(1+s^{-\frac12+\frac{3}{2p_0}})e^{-\mu's}ds\nonumber\\
	&+C\int_0^t(1+(t-s)^{-\frac12-\frac{3}{2q_0}})e^{-\lambda_1(t-s)}(1+s^{-1+\frac{3}{2q_0}})e^{-2\mu's}ds\nonumber\\
	\leq& Ce^{-\lambda_1t}+Ce^{-\mu't}(1+t^{-\frac12})\nonumber\\
	\leq& Ce^{-\mu't}
	\end{align}where we note that $p_0\in(3,\frac{3q_0}{3-q_0})$ guarantees that $\frac12+\frac{3}{2p_0}\in(0,1)$ and $\frac12-\frac{3}{2p_0}\in(0,1).$ 
	
	In the same manner as above, due to the smoothing effect of the Neumann semigroup we can further prove by standard bootstrap argument that there is $\theta_1\in(0,1)$ and $C>0$ independent of $\e$ such that for $t\geq2$ (see, e.g., \cite{Wr04})
	\begin{equation}\non
	\|\rho_\e(t)\|_{C^{\theta_1,\frac{\theta_1}{2}}(\overline{\Omega}\times[t,t+1])}+\|c_\e(t)\|_{C^{\theta_1,\frac{\theta_1}{2}}(\overline{\Omega}\times[t,t+1])}\leq C.
	\end{equation}Then applying the parabolic Schauder theory to the second equation yields a bound for $c_\e$ in $C^{2+\theta_2,1+\frac{\theta_2}{2}}(\overline{\Omega}\times[t,t+1])$ for some $\theta_2\in(0,1)$ and all $t\geq2$ which in turn indicates a bound for $\rho_\e$ in  in $C^{2+\theta_3,1+\frac{\theta_3}{2}}(\overline{\Omega}\times[t,t+1])$ for some $\theta_3\in(0,1)$ and all $t\geq3.$ This completes the proof.
\end{proof}

Now we may extract a subsequence $\e_j$ such that by passing  to the limit, we obtain a weak solution $(\rho,c)$ for the original problem \eqref{chemo1} in the sense of Definition \ref{defwk} and moreover, the preceding higher order estimates given in Lemma \ref{lm46} and the Arzel\`a--Ascoli theorem indicates that  there is $\tau_0\geq t_\eta+3$ such that $(\rho,c)\in C^{2,1}(\overline{\Omega}\times[\tau_0,\infty)$ due to uniqueness of the limit. In addition, there holds
\begin{equation}\label{upperbound}
	\sup_{t\geq\tau_0}\left(\|\rho(\cdot,t)\|_{L^\infty(\Omega)}+\|c(\cdot,t)\|_{W^{1,\infty}(\Omega)}\right)<\infty.
\end{equation}

It remains to show the convergence of the eventual smooth  solutions. In fact, since $(\rho,c)$ is smooth for  $t\geq\tau_0>0$, we can repeat the above arguments for systems \eqref{app2a}, \eqref{app2al} with $\e=0$ and recover Proposition \ref{propB} in the case $\e=0$ for $(u,v)$ instead of $(u_\e,v_\e)$.

In the same manner as before, one can first show that there exists a time $\tau_\eta\geq\tau_0$ such that 
\begin{equation}\non
\|\rho(\tau_\eta)-\M\|_{L^{3/2}(\Omega)}+\|\nabla c(\tau_\eta)\|_{L^3(\Omega)}+\|c(\tau_\eta)-\M\|_{L^\infty(\Omega)}\leq \eta<\eta_0
\end{equation}
and \eqref{help5}-\eqref{help6} hold with $t_\eta=\tau_\eta$ for all $t\geq \tau_\eta$ by replacing $(\rho_\e,c_\e)$ by $(\rho,c)$. Moreover, for $t\geq \tau_1\geq\tau_\eta+1$, there holds
\begin{equation*}
\|\rho(t)-\M\|_{L^{q_0}(\Omega)}+\|\nabla  c(t)\|_{L^{p_0}(\Omega)}\leq C e^{-\mu't},
\end{equation*}
which together with the standard $L^p-L^q$ decay estimates for the Neumann heat semigroup will  finally give rise to 
\begin{equation}\label{decay000a}
\|\rho(t)-\M\|_{L^{\infty}(\Omega)}+\|\nabla c\|_{L^\infty(\Omega)}\leq Ce^{-\mu' t},\;\;\;\|c(t)-\M\|_{L^{\infty}(\Omega)}\leq Ce^{-\min\{\mu',1\} t}\qquad\text{for}\;\;t\geq\tau_1.
\end{equation}
Indeed, introduce the reduced quantities $\rho-\M=u$ and $c-\M=v$. For simplicity, we take $\tau_\eta=0$ and hence  \eqref{help5} and \eqref{help6} hold true for $(\rho,c)$ instead of $(\rho_\e,c_\e)$ and $t_\eta=0$. Exploiting the variation-of-constants formula, the uniform boundedness \eqref{upperbound} and uniform-in-time lower bound for $c$, we deduce that
\begin{align}\label{uestimate}
&\|u(t)\|_{L^\infty(\Omega)}\non\\
\leq& \|e^{t\Delta}u_I\|_{L^\infty(\Omega)}+\int_0^t\|e^{\Delta(t-s)}\Delta v(s)\|_{L^\infty(\Omega)}ds+\int_0^t\left\|e^{\Delta(t-s)}\nabla\cdot\left(\frac{u(s)-v(s)}{c(s)}\nabla v(s)\right)\right\|_{L^\infty(\Omega)}ds\nonumber\\
\leq& k_1e^{-\lambda_1t}\|u_I\|_{L^\infty(\Omega)}+C\int_0^t(1+(t-s)^{-\frac12-\frac{3}{2p_0}})e^{-\lambda_1(t-s)}\|\nabla v(s)\|_{L^{p_0}(\Omega)}ds\nonumber\\
\leq&Ce^{-\lambda_1t}+C\int_0^t(1+(t-s)^{-\frac12-\frac{3}{2p_0}})e^{-\lambda_1(t-s)}(1+s^{-\frac12+\frac{3}{2p_0}})e^{-\mu's}ds\nonumber\\
\leq& Ce^{-\lambda_1t}+Ce^{-\mu't}\nonumber\\
\leq& Ce^{-\mu't}
\end{align}where we note that $p_0\in(3,\frac{3q_0}{3-q_0})$ guarantees that $\frac12+\frac{3}{2p_0}\in(0,1)$ and $\frac12-\frac{3}{2p_0}\in(0,1).$

On the other hand, by Lemma \ref{lmpq}-(ii)
\begin{align}\label{vestimate}
&\|\nabla v(t)\|_{L^\infty(\Omega)}\non\\
\leq&\|\nabla e^{t(\Delta-1)}v_I\|_{L^\infty(\Omega)}+\int_0^t\|\nabla e^{(\Delta-1)(t-s)}u(s)\|_{L^\infty(\Omega)}ds\nonumber\\
\leq & Ce^{-(\lambda_1+1)t}(1+t^{-\frac12})\|v_I\|_{L^\infty(\Omega)}+C\int_0^t(1+(t-s)^{-\frac12-\frac{3}{2q_0}})e^{-(\lambda_1+1)(t-s)}\|u(s)\|_{L^{q_0}(\Omega)}ds\nonumber\\
\leq&Ce^{-(\lambda_1+1)t}(1+t^{-\frac12})+C\int_0^t(1+(t-s)^{-\frac12-\frac{3}{2q_0}})e^{-(\lambda_1+1)(t-s)}e^{-\mu's}(1+s^{-1+\frac{3}{2q_0}})ds\nonumber\\
\leq& Ce^{-(\lambda_1+1)t}(1+t^{-\frac12})+Ce^{-\mu't}(1+t^{-\frac12})\non\\
\leq& Ce^{-\mu't}(1+t^{-\frac12}).
\end{align}which indicates when $t\geq1,$
\begin{equation}\label{vestimate1}
\|\nabla v(t)\|_{L^\infty(\Omega)}\leq Ce^{-\mu't}.
\end{equation}

For the convergence of $\|c(t)-\M\|_{L^\infty(\Omega)}$, we need to notice that  $\overline{c}(t)=\overline{c_I}e^{-t}+\M (1-e^{-t})$ and use the fact
\begin{align*}
\|c(t)-\M\|_{L^\infty(\Omega)}\leq& \|c(t)-\overline{c(t)}\|_{L^\infty(\Omega)}+\|\overline{c(t)}-\M\|_{L^\infty(\Omega)}\non\\
\leq& C\|\nabla c(t)\|_{L^\infty(\Omega)}+Ce^{-t}\leq C e^{-\min\{\mu',1\}t}.
\end{align*}
This completes the proof of Theorem \ref{TH2}.\qed
\begin{remark}
	The convergence rate is optimal in a sense that $\mu'$ can be chosen arbitrarily close to $\lambda_1$ which is the decay exponent of the linearized system obtained in Lemma \ref{lmexdecaypp}.
\end{remark}

\end{document}